\documentclass{article}
\usepackage[nosumlimits]{amsmath}
\usepackage{amssymb,amsthm,MnSymbol,graphicx}
\def\MM#1{\boldsymbol{#1}}
\newcommand{\pp}[2]{\frac{\partial #1}{\partial #2}} 
\newcommand{\dede}[2]{\frac{\delta #1}{\delta #2}}
\newcommand{\dd}[2]{\frac{\diff#1}{\diff#2}}

\def\MM#1{\boldsymbol{#1}}
\DeclareMathOperator{\diff}{d}
\DeclareMathOperator{\Tr}{Tr}
\DeclareMathOperator{\uu}{\MM{u}^\delta}
\DeclareMathOperator{\F}{\MM{F}^\delta}
\DeclareMathOperator{\D}{D^\delta}
\DeclareMathOperator{\q}{q^\delta}
\DeclareMathOperator{\Z}{Z^\delta}

\DeclareMathOperator{\Div}{div}
\usepackage{amscd}
\usepackage{natbib}
\usepackage{stmaryrd}
\usepackage{helvet}
\usepackage{amsfonts}
\newtheorem{theorem}{Theorem}
\newtheorem{definition}[theorem]{Definition}
\newtheorem{proposition}[theorem]{Proposition}
\newtheorem{corollary}[theorem]{Corollary}
\newtheorem{remark}[theorem]{Remark}

\usepackage[margin=2cm]{geometry}

\usepackage{color}
\newcommand{\revised}[1]{{#1}}
\usepackage{fancybox}
\begin{document}
\title{Energy-enstrophy conserving compatible finite element schemes
  for the rotating shallow water equations with slip boundary conditions}
\author{W. Bauer and C. J. Cotter}
\maketitle

\begin{abstract}
  We describe an energy-enstrophy conserving discretisation for the
  rotating shallow water equations \revised{with slip boundary
    conditions}. This relaxes the assumption of boundary-free domains
  (periodic \revised{solutions} or the surface of a sphere, for
  example) in the energy-enstrophy conserving formulation of McRae and
  Cotter (2014).  This discretisation requires extra prognostic
  vorticity variables on the boundary in addition to the prognostic
  velocity and layer depth variables. The energy-enstrophy
  conservation properties hold for any appropriate set of compatible
  finite element spaces defined on arbitrary meshes with arbitrary
  boundaries.  We demonstrate the conservation properties of the
  scheme with numerical solutions on a rotating hemisphere.
\end{abstract}

\section{Introduction}
For large scale balanced flows, energy and enstrophy are important
quantities for the rotating shallow water equations due to the cascade
of energy to large scales whilst enstrophy cascades to small scales.
At the level of numerical discretisations, energy conservation becomes
important over long time integrations, whilst enstrophy conservation
(or dissipation at the small scale) provides control of the regularity
of the velocity field over long times.

Energy and enstrophy conserving schemes for the rotating shallow water
equations have a long history that goes back to
\citet{arakawa1981potential,sadourny1975dynamics}. The finite
difference schemes in these papers were constructed from two important
ingredients: (1) the vector-invariant form of the equations, and (2)
the use of staggered grid finite difference methods built around
discretisations of the div, grad and curl operators that preserve the
vanishing div-curl and curl-grad identities at the discrete
level. These discretisations form the foundations of several
operational weather, ocean and climate models that are in current
use. Another important practical aspect is that discretisations should
preserve stationary geostrophic modes when applied to the $f$-plane
linearisation of the shallow water equations.
\citet{ringler2010unified} addressed the issue of extending these
properties to C-grid staggered finite difference discretisations on
unstructured orthogonal grids, describing separate energy-conserving
and enstrophy-conserving schemes; \citet{thuburn2012framework}
extended these ideas to non-orthogonal grids, making use of ideas from
discrete exterior calculus
\citep{hirani2003discrete}. \citet{ringler2010unified} also considered
enstrophy dissipation through the Anticipated Potential Vorticity
method, following the structured rectangular grid formulation of \citet{arakawa1990energy}. There is still no known closed form for an energy-enstrophy
conserving C-grid formulation on unstructured grids with an $f$-plane
linearisation that preserves stationary geostrophic modes, but
\citet{eldred2017total} showed that such schemes can be obtained
computationally through numerical optimisation.

In a series of papers,
\citet{salmon2004poisson,salmon2005general,salmon2007general}, Salmon
showed how to use Poisson and Nambu brackets to build conservation
into numerical discretisations. For example, \citet{stewart2016energy}
provided a C-grid discretisation for the multi-layer shallow-water
equations with complete Coriolis force. The variational formulation
finite element method makes it easier to mimic the Poisson bracket
structure of the vector-invariant shallow water equations at the
discrete level, whilst compatible finite element spaces replicate the
div-curl and curl-grad identities in the discrete
setting. \citet{mcrae2014energy} showed that this leads to a natural
energy-enstrophy conserving compatible finite element scheme, with the
bracket structure being exposed in the appendix. The finite element
exterior calculus framework underpinning these properties was exposed
by \citet{cotter2014finite}. The same structure has been exploited to
produce energy-enstrophy conserving discretisations using more exotic
finite element spaces. \citet{eldred2016high} constructed compatible
spaces from splines that allow higher-order approximations constructed
around the low-order C-grid data structure, and
\citet{lee2017discrete} used mimetic spectral elements. In the context
of imcompressible two-dimensional turbulence, \citet{natale2017scale}
considered consistent energy-conserving/enstrophy-dissipating finite
element schemes, including a formulation that extends to a consistent
energy-conserving enstrophy-dissipating version of the
\cite{mcrae2014energy} scheme, and showed that these schemes have
favourable turbulent backscatter properties.

One aspect missing from this framework is the treatment of lateral
boundaries. These are necessary for ocean modelling in the presence of
coastlines, and also in the extension to 3D vorticity conserving
schemes in the atmosphere (since there is a boundary at the Earth's
surface). \revised{In this paper,
  we are considering the inviscid equations with slip boundaries conditions. The addition of dissipative terms will introduce further boundary conditions. We do not discuss in-flow or out-flow boundary conditions.}

\citet{arakawa1990energy} \revised{avoided the consideration of
  boundary conditions} by considering layers that taper to zero depth,
allowing wetting and drying near coastlines as well as allowing layers
to outcrop the top surface in multilayer models.  This approach has
underpinned the formulation of isopycnal ocean models
\citep{hallberg1996buoyancy}.  More recently, \citep{ketefian2009mass}
produced an energy-enstrophy conserving discretisation using the
C-grid approach on structured meshes with boundaries, by making the
key observation that since vorticity cannot be diagnosed from velocity
and height at boundary points, then vorticity on the boundary should
be treated as a prognostic variable with its own conservation
equation. \citet{salmon2009shallow} used a similar idea but blended
between a vorticity-divergence-depth and velocity-depth
formulation. In this paper we show that the introduction of vorticity
degrees of freedom on the boundary can also lead to an
energy-enstrophy conserving formulation in the compatible finite
element setting.

The rest of the paper is structured as follows. In Section
\ref{sec:formulation} we review the compatible finite element
energy-enstrophy conserving formulation, to motivate the issues
relating to boundaries. We then describe the new scheme that
incorporates vorticity degrees of freedom on the boundary in order
to recover energy-enstrophy conservation when boundaries are present.
In Section \ref{sec:numerical} we demonstrate this scheme with numerical
results. Finally in Section \ref{sec:summary} we provide
a summary and outlook.

\section{Formulation}
\label{sec:formulation}
\subsection{Review of the energy-enstrophy conserving formulation
  on domains without boundaries}
The energy-enstrophy conserving formulation in \citet{mcrae2014energy}
starts from the rotating shallow-water equations in vector-invariant form
for velocity $\MM{u}$ and layer depth $D$, where
\begin{align}
  \label{eq:u}
  \MM{u}_t + \MM{F}^\perp q + \nabla\left(\frac{1}{2}|\MM{u}|^2 + gD\right) & = 0, \\
  \label{eq:D}
  D_t + \nabla\cdot\MM{F} & = 0, 
\end{align}
where $\MM{F}$ is the mass flux, and $q$ is the potential vorticity, defined
by
\begin{align}
  \label{eq:F}
  \MM{F} & = \MM{u}D, \\
  \label{eq:q}
  qD & = \nabla^{\perp}\cdot\MM{u}+f, 
\end{align}
and where $\MM{F}^\perp=\MM{k}\times \MM{F}$, $\MM{k}$ is the unit
normal to the domain surface, and $g$ is the acceleration due to
gravity.

\revised{If we consider a fluid flow where the fluid is moving in a
  container $\Omega$ such that now fluid can enter or leave, the
  relevant boundary conditions are
\begin{equation}
  \MM{u}\cdot\MM{n}=\MM{F}\cdot\MM{n}=0, \mbox{ on } \partial\Omega.
\end{equation}
There are then no boundary conditions for advected quantities such as
the layer depth $D$, since there are no advective fluxes through
$\partial\Omega$. Inviscid fluid equations such as the rotating
shallow water equations can be derived from Hamilton's principle
considering a fluid flow with slip boundary conditions, leading to
Poisson bracket formulations; for more details, see \cite{Holm98}.}

In the absence of boundaries (periodic \revised{solutions} or the
surface of a sphere being the main relevant cases), we can introduce
a weak formulation of Equations (\ref{eq:u}-\ref{eq:q}) as follows.
\begin{definition}
  \label{def:weak}
  Let $\Omega$ be a domain without boundary. We seek $\MM{u},\MM{F}\in
  H_{\Div}(\Omega)$, $D\in L^2(\Omega)$, and $q\in H^1(\Omega)$, such
  that
  \begin{align}\label{eq:u_weak}
  \langle \MM{w}, \MM{u}_t \rangle + \langle \MM{w}, q\MM{F}^\perp \rangle
  - \left\langle \nabla\cdot\MM{w}, \frac{1}{2}|\MM{u}|^2 + gD\right\rangle
  & = 0, \quad \forall \MM{w}\in H_{\Div}(\Omega), \\
  \langle \phi, D_t + \nabla\cdot\MM{F} \rangle & = 0, \quad \forall
  \phi\in L^2(\Omega),   \label{eq:D_weak}
  \\
  \langle \MM{v}, \MM{F} - \MM{u}D\rangle & = 0, \quad \forall \MM{v}\in
  H_{\Div}(\Omega), \\
  \label{eq:q_weak}
  \langle \gamma, qD \rangle + \langle \nabla^\perp\gamma, \MM{u}
  \rangle - \langle
  \gamma, f \rangle & = 0, \quad \forall \gamma\in H^1(\Omega),
  \end{align}
  where $\langle\cdot,\cdot\rangle$ is the usual $L^2$ inner product on $\Omega$.
\end{definition}
We introduce the compatible finite element spaces $(V_0, V_1, V_2)$,
that form a discrete de Rham sequence, % \werner{(i.e. $\nabla^\perp \nabla\cdot = 0$)},
\begin{equation}
\begin{CD}
  H^1 @> \nabla^\perp >> H_{\textrm{div}} @>
  \nabla\cdot
  >> L^2 \\
  @VV{\pi_0}V @VV{\pi_1}V @VV{\pi_2}V \\
  {V}_0 @> \nabla^\perp >> {V}_1 @> \nabla\cdot
  >> {V}_2 \\
\end{CD}
\end{equation}
with commuting, bounded, surjective projections $(\pi_0,\pi_1,\pi_2)$.
For the examples in this paper we have concentrated in the spaces
$V_0=CG_k$, $V_1=BDM_{k-1}$, $V_2=DG_{k-2}$. These are then used to
formulate a Galerkin finite element discretisation of Equations
(\ref{eq:u_weak}-\ref{eq:q_weak}).
\begin{definition}
  \label{def:fem no boundary}
The compatible finite element discretisation 
of the weak formulation in Definition \ref{def:weak} seeks
$(\uu,\D,\F,\q)\in
({V}_1,V_2,{V}_1,{V}_0)$ such that
\begin{align}\label{eq:u_h}
  \langle \MM{w}, \uu_t \rangle + \langle \MM{w}, \q\F^\perp \rangle
  - \left\langle \nabla\cdot\MM{w}, \frac{1}{2}|\uu|^2 + g\D\right\rangle
  & = 0, \quad \forall \MM{w}\in {V}_1, \\
  \langle \phi, \D_t + \nabla\cdot\F \rangle & = 0, \quad \forall
  \phi\in V_2,   \label{eq:D_h}
  \\
  \langle \MM{v}, \F - \uu\D\rangle & = 0, \quad \forall \MM{v}\in
  {V}_1, \\
  \label{eq:q_h}
  \langle \gamma, \q\D \rangle + \langle \nabla^\perp\gamma, \uu \rangle
   - \langle \gamma, f \rangle  & = 0, \quad \forall \gamma\in {V}_0.
\end{align}
\end{definition}
\begin{remark}
Since $D_t$ and $\nabla\cdot\F$ are both in $V_2$,
Equation \eqref{eq:D_h} is equivalent to the equation $D_t+\nabla\cdot\F=0$
holding in $L^2(\Omega)$.
\end{remark}
\begin{remark}
  It is important to note that $\q$ and {$\F$} are merely diagnostic
  variables that can be computed from the prognostic variables
  $\uu$ and {$\D$} at any time.
\end{remark}
These equations have an equivalent (almost\footnote{Almost Poisson
  brackets are anti-symmetric brackets that do not satisfy the Jacobi
  identity. This bracket satisfies the Jacobi identity on smooth
  functions but not when restricted to the finite element spaces.})
Poisson bracket formulation.
\begin{definition}[Almost Poisson bracket formulation]
Let 
$H:{V}_1\times V_2\to \mathbb{R}$ be the Hamiltonian functional, defined by
\begin{equation}
  \label{eq:H}
  H[\uu,\D] = \int_\Omega \frac{1}{2}\D|\uu|^2 + {\frac{1}{2}g\D^2} \diff x.
\end{equation}
We define the bracket $\{\cdot,\cdot\}$ by
\begin{equation}
  \left\{F,G\right\} = -\left\langle \dede{F}{\uu}, \q\dede{G}{\uu}\right\rangle
  + \left\langle \nabla\cdot \dede{F}{\uu}, \dede{G}{\D} \right\rangle
  - \left\langle \nabla\cdot \dede{G}{\uu}, \dede{F}{\D} \right\rangle.
\end{equation}
Then the corresponding almost Poisson bracket formulation defines
dynamics for any functional $F:{V}_1\times V_2\to\mathbb{R}$ by
\begin{equation}
  \dd{}{t}F[\uu,\D] = \left\{F, H\right\},
\end{equation}
where $\F$ and $\q$ are considered to be functions of $\uu$ and $\D$
defined by equations \eqref{eq:D_h} and \eqref{eq:q_h} respectively.
\end{definition}
\begin{proposition}
\label{prop:bracket}
  Equations (\ref{eq:u_h}-\ref{eq:q_h}) imply the 
bracket formulation above.
\end{proposition}
\begin{proof}
  First we compute the variational derivatives of the Hamiltonian,
  \begin{align}
    {\lim_{\epsilon \rightarrow 0} \frac{1}{\epsilon}\big( H [\uu + \epsilon \MM{w}, \D] - H [\uu , \D]\big) = }
    \left\langle \MM{w}, \dede{H}{\uu} \right\rangle
    & = \left\langle \MM{w}, \D\uu\right\rangle, %= \left\langle \MM{w}, \F\right\rangle,
    \quad \forall \MM{w}\in {V}_1, \label{eq:var_Hu}\\ %\werner{(\text{i.e.} \dede{H}{\uu} : = \F) }\\
    {\lim_{\epsilon \rightarrow 0} \frac{1}{\epsilon}\big( H [\uu, \D + \epsilon \phi  ] - H [\uu , \D]\big) = }
    \left\langle \phi, \dede{H}{\D} \right\rangle
    & = \left\langle \phi, \frac{1}{2}|\uu|^2 + g\D \right\rangle,
    \quad \forall \phi\in {V}_2.
  \end{align}
  {Hence, \eqref{eq:var_Hu} defines the discrete mass flux by $\F: = \dede{H}{\uu} \in V_1$.}
  Then we calculate
  \begin{align}
    \dot{F}[\uu,\D] & = \left\langle \dede{F}{\uu}, \uu_t \right\rangle
    + \left\langle \dede{F}{D}, D_t \right\rangle, \\
    & = \left\langle \dede{F}{\uu}, -\q\F^\perp\right\rangle
    + \left\langle \nabla\cdot\dede{F}{\uu}, \frac{1}{2}|\uu|^2 + g\D
    \right\rangle \revised{- \left\langle \dede{F}{\D}, \nabla\cdot\F \right\rangle,}\\
& = -\left\langle \dede{F}{\uu}, {\q}\dede{H}{\uu}\right\rangle
  + \left\langle \nabla\cdot \dede{F}{\uu}, \dede{H}{\D} \right\rangle
  - \left\langle \nabla\cdot \dede{H}{\uu}, \dede{F}{\D} \right\rangle
    := \left\{F,H\right\},
  \end{align}
  as required, 
  after using $\MM{w}=\dede{F}{\uu}$ and $\phi=\dede{F}{\D}$ in
  Equations (\ref{eq:u_h}-\ref{eq:D_h}).
\end{proof}
\begin{corollary}
  The Hamiltonian (equivalently, the energy) \eqref{eq:H} is conserved.
\end{corollary}
\begin{proof}
  This follows immediately from the anti-symmetry of the bracket,
  \begin{equation}
    \dot{H} = \left\{H,H\right\} = 0.
  \end{equation}
\end{proof}
\begin{remark}
  The almost Poisson formulation has an equivalent strong form formulation,
  given by
  \begin{equation}
    \pp{}{t}
    \begin{pmatrix}
    \uu \\
    \D \\
    \end{pmatrix}
    = J
    \begin{pmatrix}
      \dede{H}{\uu} \vspace{1mm} \\
      \dede{H}{\D} \\
    \end{pmatrix},
  \end{equation}
  where $J$ is an $(\uu,\D)$-dependent skew-adjoint structure
  operator defined by
  \begin{equation}
    \int_\Omega
    \begin{pmatrix}
      \MM{w} \\
      \phi \\
    \end{pmatrix}^T
    J\begin{pmatrix}
    \MM{v} \\
    \beta \\
    \end{pmatrix}
    \diff x
    = \{ L_{(\MM{w},\phi)},L_{(\MM{v},\beta)}\}, \quad \forall \MM{w},\MM{v}
    \in V_1, \, \phi,\beta \in V_2,
  \end{equation}
  and $L_{(\MM{w},\phi)}:V_1\times V_2\to \mathbb{R}$ is a functional
  defined by
  \begin{equation}
    L_{(\MM{w},\phi)}[\MM{v},\beta] = \langle \MM{w},\MM{v} \rangle
    + \langle \phi, \beta \rangle.
  \end{equation}
  This will be useful for describing energy-conserving time integration methods
  later.
\end{remark}
Equations (\ref{eq:u_h}-\ref{eq:q_h}) also conserve total potential
vorticity $Q$ and enstrophy $Z$, defined by
\begin{equation}
  Q = \int_\Omega qD\diff x, \quad Z = \int_\Omega {q^2D }\diff x.
\end{equation}
This follows directly from the implied potential vorticity
equation {as shown in the next proposition}.
\begin{proposition}
  \label{prop:q implied}
  Let $(\uu,\D,\F,\q)$ solve the compatible finite element discretisation
  in Definition \ref{def:fem no boundary}. Then $\q$ satisfies the
  equation
  \begin{equation}
    \label{eq:q_h implied}
    \langle \gamma, (\q\D)_t \rangle - \langle \nabla \gamma,
    \q\F\rangle = 0, \quad \forall \gamma \in V_0.
  \end{equation}
\end{proposition}
\begin{proof}
  Taking the time derivative of \eqref{eq:q_h}, we obtain
  \begin{equation}
    \langle \gamma, (\q\D)_t \rangle 
    = \langle -\nabla^\perp \gamma, \uu_t \rangle, \quad
    \forall \gamma \in V_0.
  \end{equation}
  If $\MM{w}=-\nabla^\perp\gamma$, then $\MM{w}\in V_1$, and
  we may use it in \eqref{eq:u_h} to obtain
  \begin{equation}
    \label{eq:q_h implied derivation}
    \langle -\nabla^\perp\gamma, \uu_t \rangle
    = \langle \nabla^\perp \gamma, \q\F^\perp \rangle
    = {\langle \nabla\gamma, \q\F \rangle,} \quad \forall \gamma \in V_0.
  \end{equation}
  Combining these two equations gives the result.
\end{proof}
\begin{remark}
  Equation \eqref{eq:q_h implied} is the standard Galerkin discretisation
  of the conservation law
  \begin{equation}
    (qD)_t + \nabla\cdot(q\MM{F}) = 0.
  \end{equation}
\end{remark}
\begin{corollary}
  $Q$ and $Z$ are conserved quantities for Equations
  (\ref{eq:u_h}-\ref{eq:q_h}). 
\end{corollary}
\begin{proof}
  We have
  \begin{equation}
    Q = \langle 1, \q\D \rangle, \quad Z = \langle \q, \q\D \rangle. 
  \end{equation}
  Since $1\in V_0$, we may use $\gamma=1$ in Equation \eqref{eq:q_h implied},
  to get
  \begin{equation}
    \dot{Q} = \langle 1, (\q\D)_t\rangle = \langle \underbrace{\nabla(1)}_{=0},
    \q\F\rangle = 0.
  \end{equation}
  Similarly, $\q\in V_0$, so we may use $\gamma=\q$ to get
{
  \begin{align}
    \dot{\Z} &= \langle \q , (\q\D)_t \rangle + \langle \q_t,\q\D \rangle, \\
    &= \langle \q , 2(\q\D)_t \rangle - \langle \q^2,D_t \rangle, \\
    &= \langle \q , 2(\q\D)_t \rangle + \langle \q^2, \nabla\cdot\F \rangle, \\
    & = \langle \nabla \q, 2\q \F \rangle+ \langle \q^2, \nabla\cdot\F \rangle, \\
    & = \left\langle \nabla (\q)^2, \F \right\rangle+ \langle \q^2,\nabla\cdot\F \rangle, \\
    & = -\left\langle (\q)^2,\nabla\cdot\F
    \right\rangle+ \langle \q^2,\nabla\cdot\F \rangle = 0.
  \end{align}
  }
\end{proof}
\begin{remark}
  Similar calculations show that $Q$ and $Z$ are Casimirs of the
  almost Poisson bracket, \emph{i.e.}
  \begin{equation}
    \{Q, F\} = \{Z, F\} = 0
  \end{equation}
  for any functional $F$.
\end{remark}
\subsection{Energy-enstrophy conserving formulation
  on domains with boundaries} We now consider the case \revised{of
  slip boundary conditions on $\partial\Omega$}. This requires us to
consider a modified de Rham complex
\begin{equation}
\begin{CD}
  \mathring{H}^1 @> \nabla^\perp >> \mathring{H}_{\textrm{div}} @>
  \nabla\cdot
  >> L^2 \\
  @VV{\pi_0}V @VV{\pi_1}V @VV{\pi_2}V \\
  \mathring{V}_0 @> \nabla^\perp >> \mathring{V}_1 @> \nabla\cdot
  >> {V}_2 \\
\end{CD}
\end{equation}
where
\begin{align}
  \mathring{H}^1 & = \left\{\psi \in H^1: \Tr_{\partial\Omega}\psi=0\right\}, \\
  \mathring{H}_{\textrm{div}} & = \left\{\MM{u} \in H_{\textrm{div}}: \Tr_{\partial\Omega}\MM{u}\cdot\MM{n}=0\right\}, \\
  \mathring{V}_0 &= \left\{\psi \in V_0: \psi=0\right\}, \\
  \mathring{V}_1 &= \left\{\MM{u} \in V_1: \MM{u}\cdot\MM{n}=0\right\},
\end{align}
and where $\Tr_{\partial\Omega}$ is the trace operator returning functions defined
in $L^2(\partial\Omega)$. The presence of $\partial\Omega$ requires the modification of \eqref{eq:q_h} to include a boundary integral,
\begin{equation}
  \label{eq:q_h bdy}
  \langle \gamma, \q\D \rangle + \langle \nabla^\perp\gamma, \uu \rangle
  - \llangle \gamma, \MM{n}^\perp\cdot\uu \rrangle
   - \langle \gamma, f \rangle = 0, \quad \forall \gamma\in {V}_0,
\end{equation}
where $\llangle\cdot,\cdot\rrangle$ defines the $L^2$ inner product on
$\partial\Omega$,
\begin{equation}
  \llangle f, g \rrangle = \int_{\partial\Omega} fg \diff x.
\end{equation}
If we just apply this modification to the discretisation in Definition
\ref{def:fem no boundary}, replacing $V_1$ with $\mathring{V}_1$ and
$V_2$ with $\mathring{V}_2$, then we still have an almost Poisson
bracket formulation (so energy is still conserved). Proposition
\ref{prop:q implied} does not hold though, because we can only take
$w=\nabla^\perp\gamma$ in \eqref{eq:q_h implied derivation} if
$\gamma\in \mathring{V}_0$, but we need to be able to take $\gamma \in
V_0$. In particular, the test functions $1$ and $\q$ required to show
conservation of total potential vorticity and enstrophy are not in
$\mathring{V}_0$ (at least, not in general for $\q$). In numerical
experiments with this formulation we do indeed see unbounded growth in
enstrophy due to sources at the boundary which eventually pollute the
solution throughout the domain.

To resolve this problem, we introduce a vorticity variable $\Z\in V_0$,
such that
\begin{equation}
  \int_\Omega \gamma \Z \diff x = \int_\Omega \gamma \q \D \diff x, \quad
  \forall \gamma \in V_0.
\end{equation}
The projection $\mathring{Z}_0$ of $\Z$ into $\mathring{V}_0$, is thus a 
{diagnostic} variable that can be obtained from $\uu$ according to
\begin{equation}
  \langle \gamma, \mathring{Z} \rangle + \langle \nabla^\perp\gamma, \uu \rangle
   - \langle \gamma, f \rangle = 0, \quad \forall \gamma\in \mathring{V}_0,
\end{equation}
since now $\nabla^\perp\gamma\in \mathring{V}_1$. However, to compute $\Z$
we also need to know its projection $Z'$ onto
$\mathring{V}_0^\perp$, defined as the $L^2$-orthogonal complement of
$\mathring{V}_0$ in $V_0$. $Z'$ may be initialised by obtaining $\Z\in V_0$ from
\begin{equation}
  \langle \gamma, \Z \rangle + \langle \nabla^\perp\gamma, \uu \rangle
  - \llangle \gamma, \MM{n}^\perp\cdot \uu \rrangle
   - \langle \gamma, f \rangle = 0, \quad \forall \gamma\in {V}_0,
\end{equation}
before projecting to $\mathring{V}_0^\perp$ to obtain $Z'$. After initialisation, $Z'$ has its own dynamics as given below.
\begin{definition}
  \label{def:fem boundary}
The compatible finite element discretisation 
of the rotating shallow water equations {seeks} \linebreak
$(\uu,\D,Z',\mathring{Z},\F,\q)\in
(\mathring{V}_1,V_2,\mathring{V}_0^\perp,\mathring{V}_0,\mathring{V}_0,{V}_0)$ such that
\begin{align}
  \label{eq:bcs u}
  \langle \MM{w}, \uu_t \rangle + \langle \MM{w}, \q\F^\perp \rangle
  - \left\langle \nabla\cdot\MM{w}, \frac{1}{2}|\uu|^2 + g\D\right\rangle
  & = 0, \quad \forall \MM{w}\in \mathring{V}_1, \\
  \langle \phi, \D_t + \nabla\cdot\F \rangle & = 0, \quad \forall
  \phi\in V_2,
  \label{eq:bcs D}
  \\ \label{eq:bcs zprime}
  \langle \gamma,Z'_t \rangle - \langle \nabla \gamma, \F \q \rangle &= 0, \quad \forall \gamma \in \mathring{V}_0^\perp, \\
  \label{eq:bcs zring}
  \langle \gamma, \mathring{Z} \rangle + \langle \nabla^\perp\gamma, \uu \rangle
 - \langle \gamma, f \rangle   & = 0, \quad \forall \gamma\in \mathring{V}_0, \\
  \label{eq:bcs F}
  \langle \MM{v}, \F - \uu\D\rangle & = 0, \quad \forall \MM{v}\in
  \mathring{V}_1, \\
  \langle \gamma, \q\D - \mathring{Z} - Z'\rangle  & = 0, \quad \forall \gamma
  \in V_0.
  \label{eq:bcs qZ'}
\end{align}
\end{definition}
\revised{The dynamics of $Z'$ are precisely chosen so as to recover
  the correct dynamics for potential vorticity $q^\delta\in V_0$ (and
  not just the projection of potential vorticity into
  $\mathring{V}_0$), \emph{i.e.} so that Proposition \ref{prop:q implied
    bcs} will hold. This idea is analogous to \cite{ketefian2009mass},
  who introduced vorticity variables at boundary vertices that have
  dynamics that imply the PV equation at the boundary, which is
  otherwise undefined.}

Since $Z'$ is independent of $\uu$ and $\D$, we need to enlarge the phase
space to obtain the bracket formulation for this extended system to include it.
\begin{definition}[Extended almost Poisson bracket formulation]
\label{def:extended bracket}
  Let 
$H:\mathring{V}_1\times V_2\times \mathring{V}_0^\perp\to \mathbb{R}$ be the Hamiltonian functional, defined by
\begin{equation}
  H[\uu,\D, Z'] = \int_\Omega \frac{1}{2}\D|\uu|^2 + {\frac{1}{2}g\D^2}\diff x.
\end{equation}
We define the bracket $\{\cdot,\cdot\}$ by
\begin{equation}
  \left\{F,G\right\} = -\left\langle \dede{F}{\uu}, \q\dede{G}{\uu}\right\rangle
  + \left\langle \nabla\cdot \dede{F}{\uu}, \dede{G}{\D} \right\rangle
  - \left\langle \nabla\cdot \dede{G}{\uu}, \dede{F}{\D} \right\rangle
  + \left\langle \nabla \dede{F}{Z'}, \q\dede{G}{\uu} \right \rangle
  - \left\langle \nabla \dede{G}{Z'}, \q\dede{F}{\uu} \right \rangle.
\end{equation}
Then the corresponding almost Poisson bracket formulation defines
dynamics for any functional $F:\mathring{V}_1\times V_2\times \mathring{V}_0^\perp\to\mathbb{R}$ by
\begin{equation}
  \dd{}{t}F[\uu,\D, Z'] = \left\{F, H\right\},
\end{equation}
where $\mathring{Z}$, $\F$ and $\q$ are considered to be functions of $\uu$ and $\D$
defined by the equations (\ref{eq:bcs zring}-\ref{eq:bcs qZ'}) respectively.
\end{definition}
\begin{proposition}
  Equations (\ref{eq:bcs u}-\ref{eq:bcs qZ'}) imply the 
bracket formulation in Definition \ref{def:extended bracket}.
\end{proposition}
\begin{proof}
  The proof is similar to the proof of Proposition \ref{prop:bracket},
  except that we need to also check the dynamics for $Z'$. Taking
  $F=\langle \gamma, Z' \rangle$ for $\gamma\in \mathring{V}_0^\perp$,
  {so that $\dede{F}{Z'} = \gamma$}, we have
  \begin{equation}
    \dot{F} = \langle \gamma, \dot{Z}' \rangle = \{F,H\}=
    \langle \nabla{\gamma}, \q\dede{H}{\uu}\rangle
    = \langle \nabla {\gamma}, \q \F \rangle, 
  \end{equation}
  as required.
\end{proof}
\begin{corollary}
  Equations (\ref{eq:bcs u}-\ref{eq:bcs qZ'}) have conserved energy.
\end{corollary}
\begin{proof}
  Follows directly from the almost Poisson bracket formulation.
\end{proof}
\begin{proposition}
  \label{prop:q implied bcs}
  Let $(\uu,\D,\F,\q)$ solve the compatible finite element discretisation
  in Definition \ref{def:fem boundary}. Then $\q$ satisfies the
  equation
  \begin{equation}
    \langle \gamma, (\q\D)_t \rangle - \langle \nabla \gamma,
    \q\F\rangle = 0, \quad \forall \gamma \in V_0.
  \end{equation}
\end{proposition}
\begin{proof}
  Taking the time derivative of \eqref{eq:bcs qZ'}, we obtain
  \begin{align}
   \langle \gamma, (\q\D)_t \rangle 
    & = \langle \gamma, \mathring{Z}_t + Z'_t \rangle, \\
    & = \langle \mathring{\gamma},\mathring{Z}_t \rangle
    + \langle \gamma', {Z}_t' \rangle, \\
\mbox{Equations \eqref{eq:bcs zring} and \eqref{eq:bcs zprime}}\quad   & = \langle -{\nabla^\perp}\mathring{\gamma}, \uu_t \rangle
    + \langle \nabla\gamma', \F \q \rangle \\
\mbox{Equation \eqref{eq:bcs u}}\quad    & = \langle {\nabla^\perp}\mathring{\gamma}, \q\F^\perp \rangle
    + \langle \nabla\gamma', \F \q \rangle \\
    & = \langle \nabla (\mathring{\gamma}+\gamma'), \F\q \rangle = \langle \nabla \gamma, \F\q \rangle
    \quad \forall \gamma \in V_0,
  \end{align}
  as required, where $\mathring{\gamma}$ is the $L^2$ projection of
  $\gamma$ into $\mathring{V}_0$, and $\gamma'$ is the $L^2$
  projection of $\gamma'$ into $\mathring{V}_0^\perp$, so that $\gamma=\mathring{\gamma}+\gamma'$.
\end{proof}
\begin{corollary}
  The total potential vorticity $Z$ and potential enstrophy $Q$ are conserved
  by the discretisation in Definition \ref{def:fem boundary}.
\end{corollary}
\begin{proof}
  The proof follows from the implied PV equation, as for the
  boundary-free case.
\end{proof}

It is inpractical to deal with $\mathring{V}_0^\perp$ as there is no
local basis. However, an equivalent formulation exists.
\begin{definition}[Equivalent extended formulation]
\label{def:equivalent}
  We seek
$(\uu,\D,\F,\q)\in
(\mathring{V}_1,V_2,\mathring{V}_1,{V}_0)$ such that
  \begin{align}
    \label{eq:equivalent u}
  \langle \MM{w}, \uu_t \rangle + \langle \MM{w}, \q\F^\perp \rangle
  - \left\langle \nabla\cdot\MM{w}, \frac{1}{2}|\uu|^2 + g\D\right\rangle
  & = 0, \quad \forall \MM{w}\in \mathring{V}_1, \\
  \langle \phi, \D_t + \nabla\cdot\F \rangle & = 0, \quad \forall
  \phi\in V_2,
  \label{eq:equivalent D}
  \\
  \langle \MM{v}, \F - \uu\D\rangle & = 0, \quad \forall \MM{v}\in
  \mathring{V}_1,
  \label{eq:equivalent F}
  \\
  \langle \gamma, (\q\D)_t \rangle - \langle \nabla \gamma,
  \F\q \rangle & = 0, \quad \forall \gamma
  \in V_0. \label{eq:equivalent (qD)_t}
\end{align}
\end{definition}
\begin{proposition}
  If $\q$ is initialised to satisfy \eqref{eq:q_h bdy} then solving
  the equations in Definition \ref{def:equivalent} is equivalent to solving
  the equations in Definition \ref{def:fem boundary}.
\end{proposition}
\begin{proof}
  Proposition \ref{prop:q implied bcs} means that solving the equations
  in Definition \ref{def:fem boundary} produces a solution to
  the equations in Definition \ref{def:equivalent}. It remains to check
  the converse.

  Let $\uu$, $\D$, $\q$ solve the equations in Definition
  \ref{def:equivalent}, with $\q$ is initialised to satisfy
  \eqref{eq:q_h bdy}.  Equations \eqref{eq:bcs u}, \eqref{eq:bcs D},
  and \eqref{eq:bcs F} also appear in Definition \ref{def:fem
    boundary}, so it remains to check that Equations \eqref{eq:bcs
    zprime}, \eqref{eq:bcs zring} and \eqref{eq:bcs qZ'} are satisfied.
  We define $\mathring{Z}\in \mathring{V}_0$ and $Z'\in \mathring{V}_0^{\perp}$
  according to
  \begin{align}
    \int_\Omega \gamma \mathring{Z} \diff x & = \int_\Omega \q\D \gamma \diff x, \quad \forall \gamma \in \mathring{V}_0, \\
    \int_\Omega \gamma Z' \diff x & = \int_\Omega \q \D \gamma \diff x, \quad
    \forall \gamma \in \mathring{V}_0^\perp.
  \end{align}
  Since $\mathring{V}_0$ and $\mathring{V}_0^\perp$ are orthogonal in
  $L^2(\Omega)$, we obtain $\Z=Z'+\mathring{Z}$ from which 
  \eqref{eq:bcs qZ'} follows.

  Then taking 
  {$\gamma\in \mathring{V}_0^\perp$ in \eqref{eq:equivalent (qD)_t}} recovers
  \eqref{eq:bcs zprime}. If we initialise according to \eqref{eq:q_h bdy},
  then \eqref{eq:bcs zring} is satisfied initially, and it remains to
  check that
  \begin{equation}
    \pp{}{t}\left(\langle \gamma, \mathring{Z}\rangle + \langle
    \nabla^\perp \gamma, \uu \rangle\right) = 0, \quad \forall \gamma \in
    \mathring{V}_0.
  \end{equation}
  We obtain this by time differentiating the definition of $\mathring{Z}$
  as the projection of $\q\D$ into $\mathring{V}_0$ to obtain
  \begin{align}
    \pp{}{t}\langle \gamma, \mathring{Z} \rangle & =
    \pp{}{t} \langle \gamma, \q \D \rangle, \\
   \mbox{Equation \eqref{eq:equivalent (qD)_t}}\quad & = -\langle \nabla \gamma, \F \q \rangle, \\
    & = -\langle \nabla^\perp\gamma, \uu \rangle, \quad
    \forall \gamma\in\mathring{V}_0,
  \end{align}
  as required.
\end{proof}
\begin{definition}[Poisson brackets for equivalent extended
    formulation]
  \label{def:equivalent poisson}
    Let 
$H:\mathring{V}_1\times V_2\times{V}_0\to \mathbb{R}$ be the Hamiltonian functional, defined by
\begin{equation}
  H[\uu,\D, \Z] = \int_\Omega \frac{1}{2}\D|\uu|^2 + {\frac{1}{2}g\D^2} \diff x.
\end{equation}
We define the bracket $\{\cdot,\cdot\}$ by
\begin{equation}
  \left\{F,G\right\} = -\left\langle \dede{F}{\uu}, \q\dede{G}{\uu}\right\rangle
  + \left\langle \nabla\cdot \dede{F}{\uu}, \dede{G}{\D} \right\rangle
  - \left\langle \nabla\cdot \dede{G}{\uu}, \dede{F}{\D} \right\rangle
  + \left\langle \nabla\dede{F}{\Z}, \dede{G}{\uu}\q \right \rangle
  - \left\langle \nabla\dede{G}{\Z}, \dede{F}{\uu}\q \right \rangle,
\end{equation}
where $\q$ is understood to be a function of $\Z$ and $\D$ given
by
\begin{equation}\label{eq:def_Z}
  \langle \gamma, \q \D \rangle = \langle \gamma, \Z \rangle, \quad
  \forall \gamma \in V_0,
\end{equation}
and where  $\F$ is considered to be a function of $\uu$ and $\D$
defined by equation (\ref{eq:bcs F}).
Then the corresponding almost Poisson bracket formulation defines
dynamics for any functional $F:\mathring{V}_1\times V_2\times V_0\to\mathbb{R}$ by
\begin{equation}
  \dd{}{t}F[\uu,\D, Z] = \left\{F, H\right\}.
\end{equation}
\end{definition}
\begin{proposition}
  Equations (\ref{eq:equivalent u}-\ref{eq:equivalent (qD)_t}) imply the 
bracket formulation in {Definition \ref{def:equivalent poisson}.}
\end{proposition}
\begin{proof}
  As before it remains to check the dynamics for $\q$. Taking $F=
  \langle\gamma,\Z\rangle$ for $\gamma \in V_0$, and noting
  $\dede{H}{\Z}=0$, we get
  \begin{align}
    \langle \gamma, \Z_t \rangle &= \dot{F} = \{F, H\} \\
    & = \langle \nabla\gamma, \F\q \rangle.
  \end{align}
  On the other hand, {time differentiating \eqref{eq:def_Z} gives}
  \begin{align}
    \langle \gamma, (\q\D)_t \rangle = \langle \gamma, \Z_t \rangle,
  \end{align}
  hence the result.
\end{proof}
\section{Numerical results}
\label{sec:numerical}
\subsection{Energy-conserving time integration}
In this section we demonstrate the numerical scheme, using the
energy-preserving Poisson integrator of \cite{cohen2011linear}. Whilst
the implicit midpoint rule conserves quadratic Hamiltonians exactly,
the energy-preserving Poisson integrator extends this property to
higher-degree polynomials, such as the shallow-water Hamiltonian which
is cubic. Given a Poisson system of the form
\begin{equation}
  \dot{\MM{z}} = J(\MM{z})\dede{}{\MM{z}}H(\MM{z}),
\end{equation}
the Poisson integrator takes the form
\begin{equation}
  \MM{z}^{n+1} = \MM{z}^n + \Delta t J\left(\frac{\MM{z}^n + \MM{z}^{n+1}}{2}\right)
  \overline{\dede{}{\MM{z}}H}, 
  \quad \mbox{where }\overline{\dede{}{\MM{z}}H} = 
  \int_{0}^{1} \dede{}{\MM{z}} H(\MM{z}^n + s(\MM{z}^{n+1}-\MM{z}^n))\diff s.
\end{equation}
For our equivalent extended shallow water discretisation, we have
\begin{align}
  \langle \phi, \overline{\dede{H}{\D}}\rangle &=
%  \langle \phi, \int_0^1 g(D^n + \tau (D^{n+1} - D^n) + \frac{1}{2}
%  \left(|\MM{u}^n + \tau (\MM{u}^{n+1} - \MM{u}^n)|^2\right)
%  \diff \tau, \\
%  & = \langle \phi, g \frac{1}{2}\left(D^n + D^{n+1}\right)
%  + \frac{1}{2}\int_{t^n}^{t^{n+1}}\left(|\MM{u}^n|^2 + 2\tau\MM{u}^n \cdot (\MM{u}^{n+1}-\MM{u}^n)
%  + \tau^2\left(|\MM{u}^n|^2 - 2 \MM{u}^{n+1}\cdot\MM{u}^n + |\MM{u}^{n+1}|^2\right)\right)
%  \rangle \diff \tau, \\
%  & = \langle \phi, g \frac{1}{2}\left(D^n + D^{n+1}\right)
%  + \frac{1}{2}\left(|\MM{u}^n|^2 + \MM{u}^n \cdot (\MM{u}^{n+1}-\MM{u}^n)
%  + \frac{1}{3}\left(|\MM{u}^n|^2 - 2 \MM{u}^{n+1}\cdot\MM{u}^n + |\MM{u}^{n+1}|^2\right)\right)
%  \rangle, \\
   \langle \phi, g \frac{1}{2}\left(\D^n + \D^{n+1}\right)
  + \frac{1}{3}\left(|\uu^n|^2 + \uu^n \cdot \uu^{n+1} + |\uu^{n+1}|^2\right)
  \rangle,  \quad \forall \phi \in V_2, \\
  \langle \MM{w}, \overline{\dede{H}{\uu}} \rangle & =
%  \int_{t^n}^{t^{n+1}}\langle \MM{w}, \left(D^n + \tau(D^{n+1}-D^n)\right)
%  \left(\MM{u}^n + \tau(\MM{u}^{n+1}-\MM{u}^n)\right)\rangle \diff \tau, \\
%   & =
%  \int_{t^n}^{t^{n+1}}\langle \MM{w}, D^n\MM{u}^n + \tau\left(D^n\left(\MM{u}^{%n+1}
%  -\MM{u}^n\right)
%  + \left(D^{n+1}-D^n\right)\MM{u}^n\right) + \tau^2\left(D^{n+1}-D^n\right)
%  \left(\MM{u}^{n+1}-\MM{u}^n\right) \rangle \diff \tau,\\
%   & =
%  \langle \MM{w}, D^n\MM{u}^n + \frac{1}{2}\left(D^n\left(\MM{u}^{n+1}
%  -\MM{u}^n\right)
%  + \left(D^{n+1}-D^n\right)\MM{u}^n\right) + \frac{1}{3}\left(D^{n+1}-D^n\righ%t)
%  \left(\MM{u}^{n+1}-\MM{u}^n\right) \rangle,\\
%  & =
%  \langle \MM{w}, \frac{1}{2}\left(D^n\MM{u}^{n+1}
%  + D^{n+1}\MM{u}^n\right) + \frac{1}{3}\left(D^{n+1}\MM{u}^{n+1}
%  + D^n\MM{u}^n - D^n\MM{u}^{n+1} - D^{n+1}\MM{u}^n\right)\rangle, \\
%  & =
  \frac{1}{3}\langle \MM{w}, \D^n\uu^n + \frac{1}{2}\D^n\uu^{n+1}
  + \frac{1}{2}\D^{n+1}\uu^n + \D^{n+1}\uu^{n+1}\rangle, \quad
  \forall \MM{w}\in \mathring{V}_1.
\end{align}
Hence, we obtain the following time discretisation,
\begin{align}
  \label{eq:time u}
  \langle \MM{w}, \uu^{n+1} - \uu^n + \Delta t \overline{\dede{H}{\uu}}^\perp
  \left(\frac{\q^{n+1}+\q^n}{2}\right) \rangle - \Delta t\langle \nabla\cdot \MM{w}, \overline{\dede{H}{\D}}
  \rangle & = 0, \quad \forall \MM{w} \in \mathring{V}_1, \\
  \label{eq:time D}
  \langle \phi, \D^{n+1} - \D^n + \Delta t \nabla\cdot \overline{\dede{H}{\uu}}
  \rangle & = 0, \quad \forall \phi \in V_2, \\
  \langle \gamma, \Z^{n+1} - \Z^n\rangle -\Delta t
  \langle \nabla\gamma, \overline{\dede{H}{\uu}} \rangle \left(\frac{\q^{n+1}+\q^n}{2}\right)& = 0, \quad \forall
  \gamma \in V_0.
\end{align}
This is equivalent to solving for
$(\uu^{n+1},\D^{n+1},\F^{n+1/2},\q^{n+1}) \in
\mathring{V}_1\times V_2 \times \mathring{V}_1 \times V_0$,
such that 
\begin{align}
  \langle \MM{w}, \uu^{n+1} - \uu^n + \Delta t {\F^{n+1/2}}^\perp
  \left(\frac{\q^{n+1}+\q^n}{2}\right) \rangle - \Delta t\langle \nabla\cdot \MM{w}, g\frac{\D^{n+1}+\D^n}{2}
  + \overline{K} \rangle  & = 0, \quad \forall \MM{w} \in \mathring{V}_1, \\
  \langle \phi, \D^{n+1} - \D^n + \Delta t \nabla\cdot {\F^{n+1/2}}
  \rangle & = 0, \quad \forall \phi \in V_2, \\
  \label{eq:time F}
  \langle \MM{v}, \F^{n+1/2} - \overline{\F} \rangle & = 0, \quad
  \forall \MM{v} \in \mathring{V}_1, \\
  \label{eq:time q}
  \langle \gamma, \q^{n+1}\D^{n+1} - \q^n\D^n\rangle -\Delta t
  \left\langle \nabla\gamma, {\F^{n+1/2}}\frac{\q^{n+1}+\q^n}{2}
  \right\rangle & = 0, \quad \forall
  \gamma \in V_0,
\end{align}
where
\begin{align}
  \overline{\F} &= \uu^{n+1}\D^{n+1}/3 + \uu^n\D^{n+1}/6 + \uu^{n+1}\D^n/6
  + \uu^n\D^n/3, \\
  \overline{K}& = |\uu^{n+1}|^2/3 + \uu^n\cdot\uu^{n+1}/3
  + |\uu^n|^2/3.
\end{align}

\begin{figure}
  \centerline{\includegraphics[width=8cm]{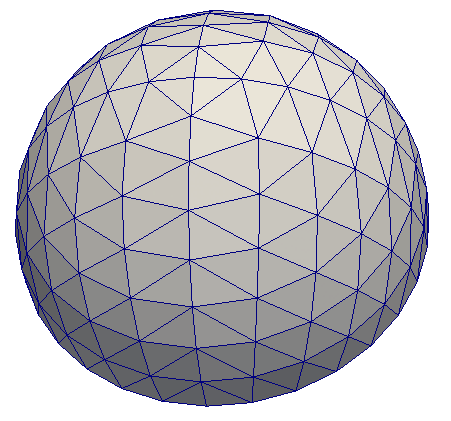}}
  \caption{\label{fig:spheremesh}Refinement level 3 of the octahedral hemisphere mesh used in the numerical examples.}
\end{figure}

In this section, we demonstrate the energy-enstrophy conserving scheme
through the adaptation of two of the popular shallow water sphere
testcases \citep{williamson1992standard} to the hemisphere domain, so
that the equator becomes a boundary. In all cases the mesh used is
(half of) an octahedral mesh obtained by hierarchical refinement of
the triangular faces of an octahedron, before mapping to the sphere so
that lines of constant height are mapped to lines of constant latitude
(see Figure \ref{fig:spheremesh}). We use the compatible finite
element spaces as follows: $P3$ for vorticity, $BDM2$ for velocity,
$P1_{DG}$ for layer depth.  
All numerical calculations are performed
using Firedrake, the automated code generation finite element library
\citep{rathgeber2016firedrake}. In all cases we use a sphere of radius
$R_0 = 6371220m$, with rotation rate $\Omega=7.292\times 10^{-5}s^{-1}$ so
that $f=2\Omega z/R_0$, and graviational acceleration
$g=9.810616ms^{-2}$.

First we consider a convergence test based on Test Case 2, which is a
steady state solid rotation solution, which is also a solution when
restricted to the Northern hemisphere, because the velocity is tangential
to the equator, and the solution is zonally symmetric. The velocity
and height are intialised according to the steady state solution
\begin{align}
  \MM{u} & = -u_0\left(-y, x, 0\right)/R_0, \\
  D & = h_0 - (R_0\Omega u_0 + u_0^2/2.0)\frac{z^2}{R_0^2g}, 
\end{align}
where $h_0=5960m$ and $u_0=2\pi R_0/(12\mbox{ days })$. The equations
are then integrated numerically for 15 days and the solution fields
are compared against the initial conditions. We might expect any
issues associated with consistency errors at the boundary to manifest
themselves as loss of optimal convergence rates in the $L^2$ norm of
these errors; the convergence rates (see Figure \ref{fig:w2 convergence}) match those of the scheme applied
to the full sphere. In these
calculations the relative energy is conserved to 7 decimal places,
although we do observe some small fluctuations below that level of
precision which are due to the tolerance settings of the Newton
solver. As is expected for a steady state solution, we also observe
conservation of relative enstrophy to 6 decimal places even though
conservation is not guaranteed by the time integration method. Plots
of energy and enstrophy for this experiment are shown in Figure
\ref{fig:w2enens}.

\begin{figure}
  \centerline{
  \includegraphics[width=12cm]{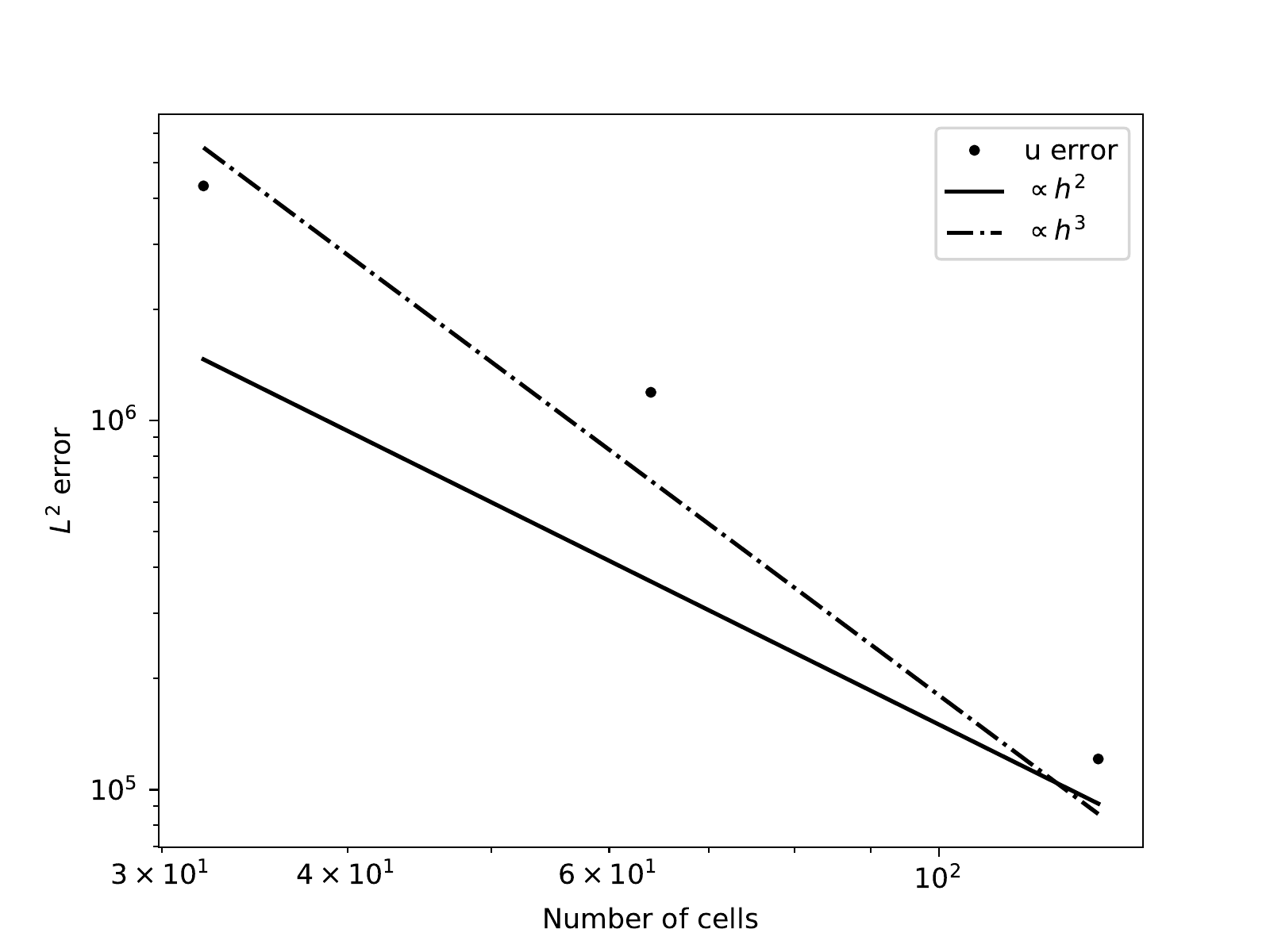}}
  \centerline{
  \includegraphics[width=12cm]{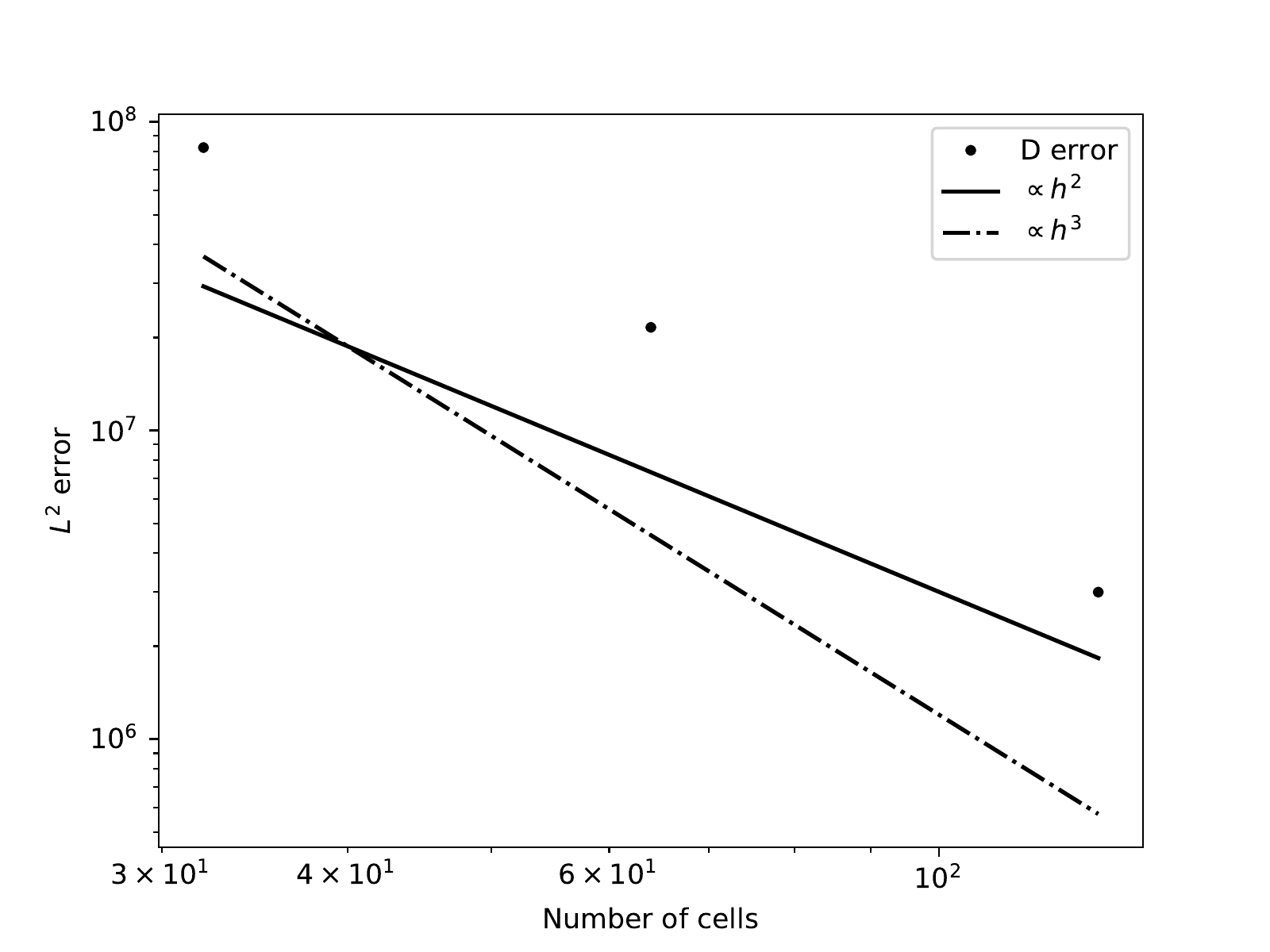}
  }
  \caption{\label{fig:w2 convergence}Plots showing $L^2$ errors
    at three different resolutions (mesh refinement levels 3, 4 and 5)
    in velocity and layer depth. We observe superconvergent results
    in both cases (we would expect second-order convergence for this
    set of finite element spaces), which is probably due to the
    symmetries of the mesh.}
\end{figure}

\begin{figure}
  \centerline{
  \includegraphics[width=12cm]{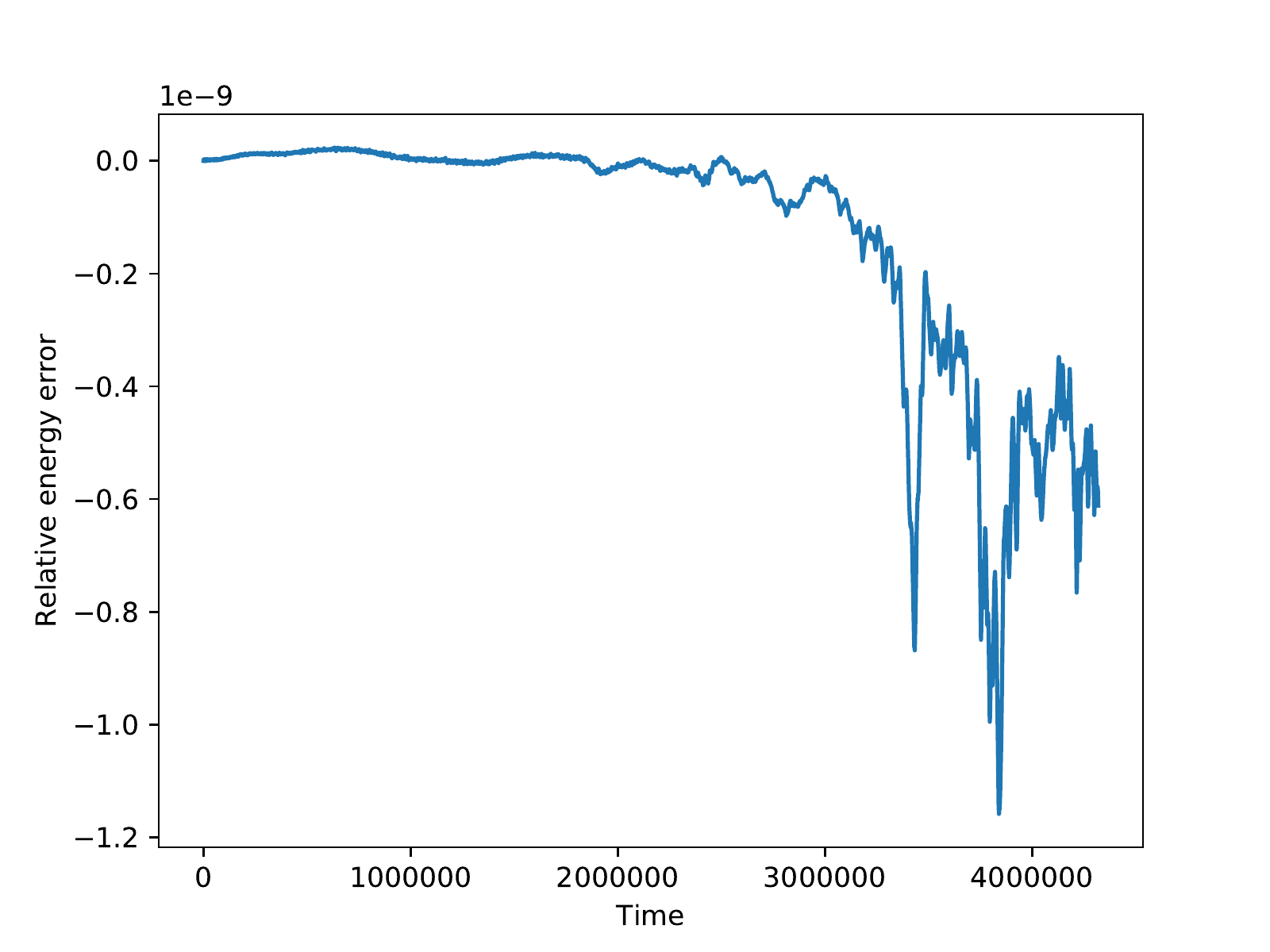}}
  \centerline{
  \includegraphics[width=12cm]{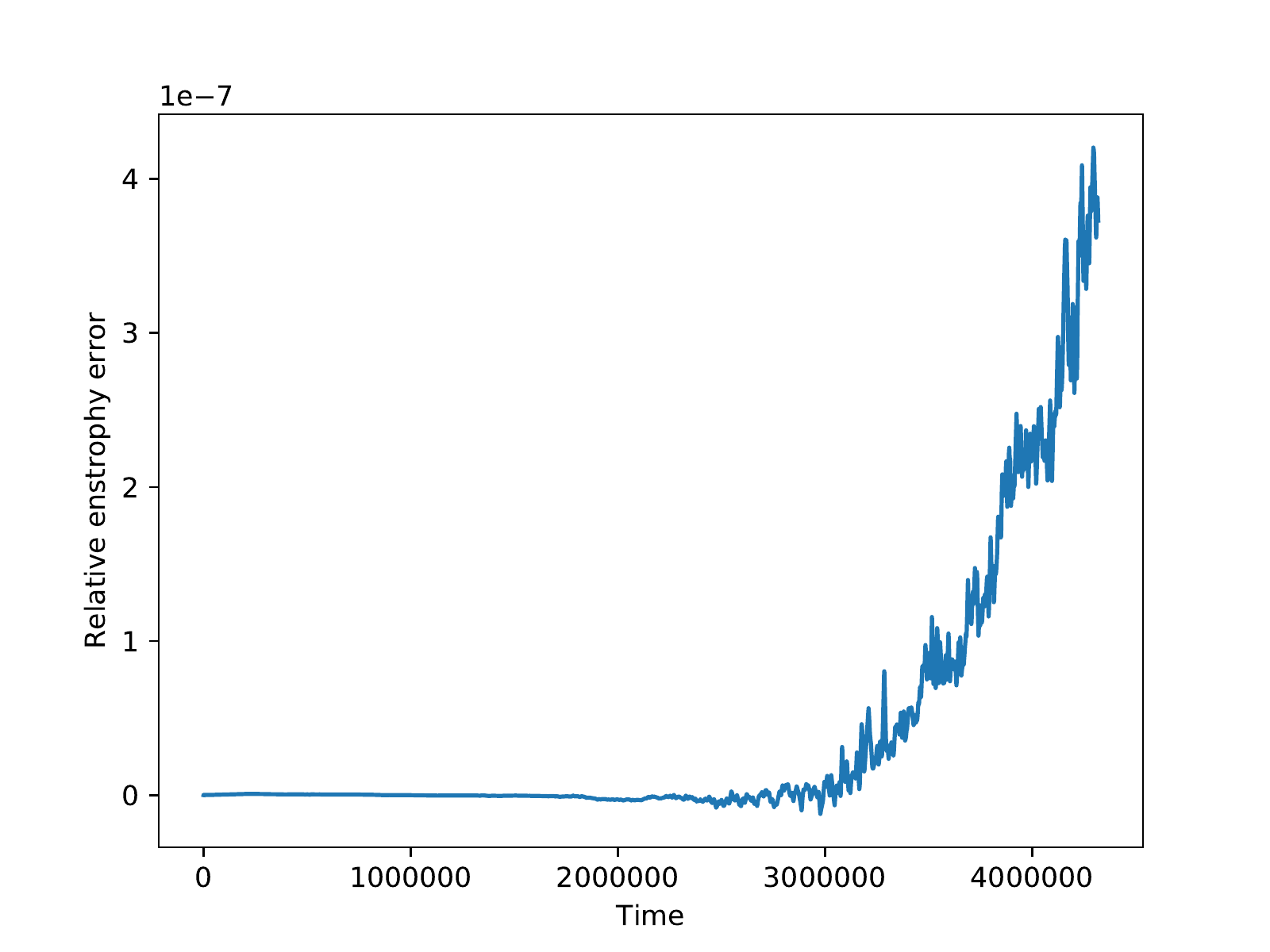}
  }
  \caption{\label{fig:w2enens}Plots showing relative energy and
    enstrophy errors for the Williamson 2 test case at mesh refinement
    level 5. We observe energy conservation up to solver tolerance and
    approximate enstrophy conservation for long times.}
\end{figure}

\revised{To verify that the correct boundary condition is being enforced, we
ran a Kelvin wave testcase in a disk with constant $f$. There are no
analytic solutions for this problem, either for nonlinear or linear
equations, but it is possible to observe that the Kelvin wave
propagation mechanism is functioning correctly. In the limit of large
radius compared to the Rossby deformation radius and the limit of low
amplitude, the solution should approach that of the linear Kelvin wave
along a straight boundary, which propagates at speed $\sqrt{gH}$. We
verified that this is approximately the case by considering a disk of
radius 1, $H=g=1$ so that the wave speed is 1, and $f=10$ so that the deformation radius is 1/10, with initial
data
\begin{equation}
  \MM{u} = \MM{e}_\theta a_0 \exp((r-1)f)y, \quad
  h = a_0\exp((r-1)f)y,
\end{equation}
where $a_0=0.01$ is the wave amplitude. The results of this testcase
are shown in Figure \ref{fig:kelvin}.

\begin{figure}
    \centerline{
    \includegraphics[width=5cm]{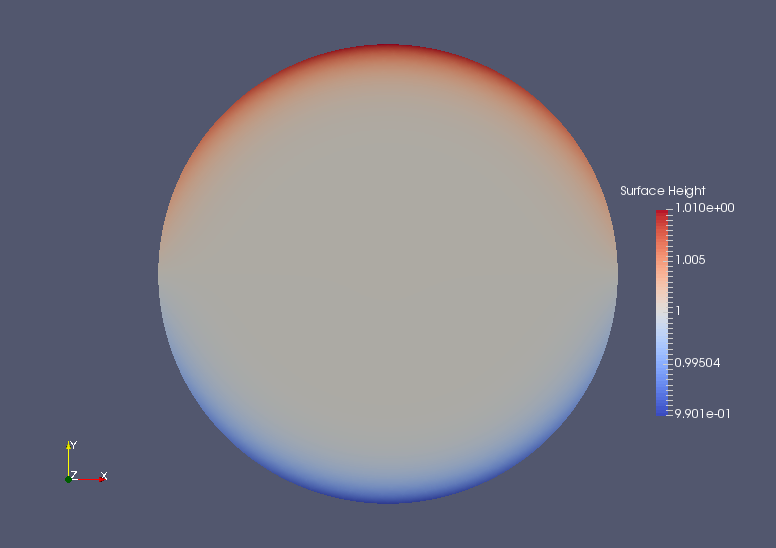}
    \includegraphics[width=5cm]{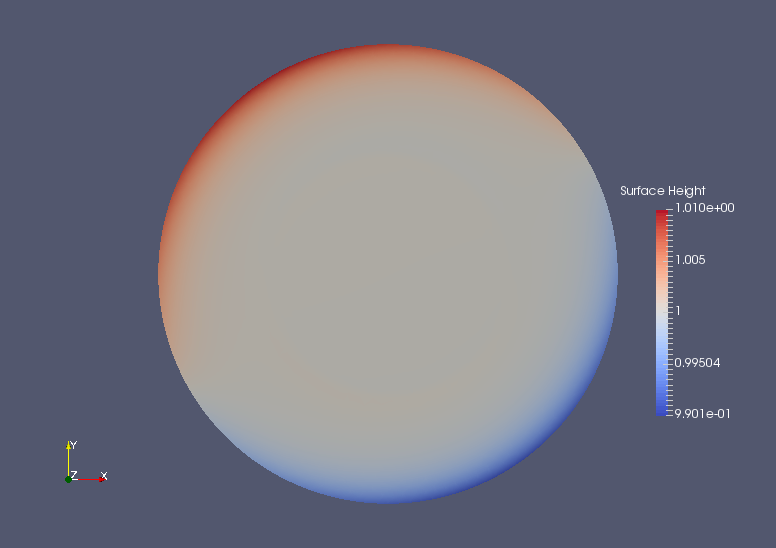}
    \includegraphics[width=5cm]{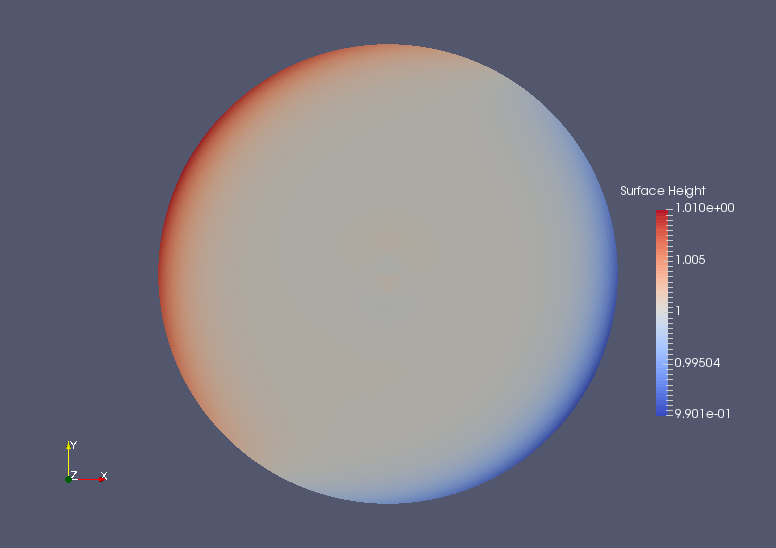}
    }
  \centerline{
    \includegraphics[width=5cm]{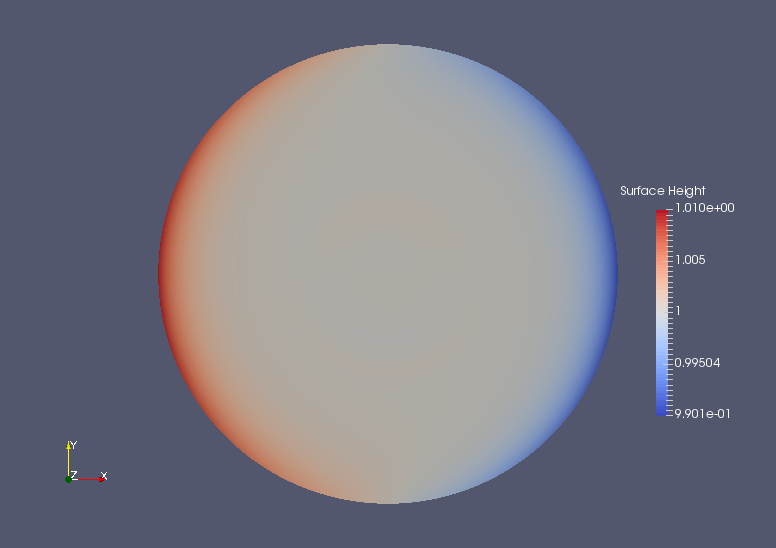}
    \includegraphics[width=5cm]{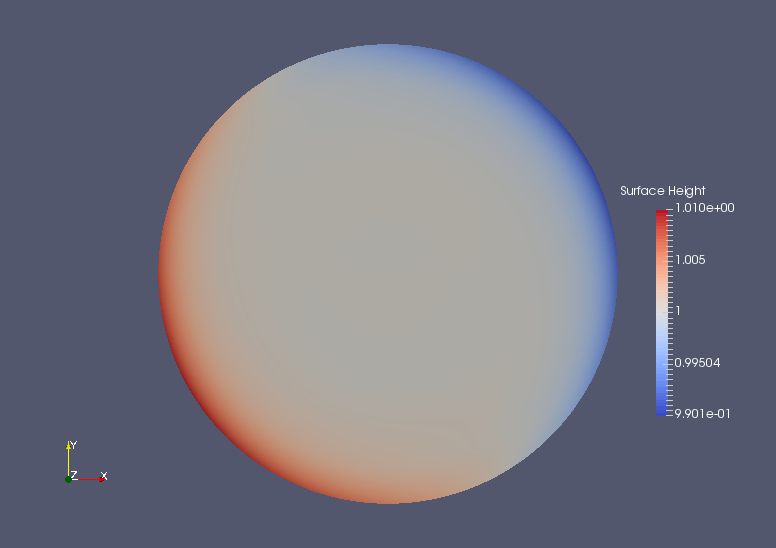}
    \includegraphics[width=5cm]{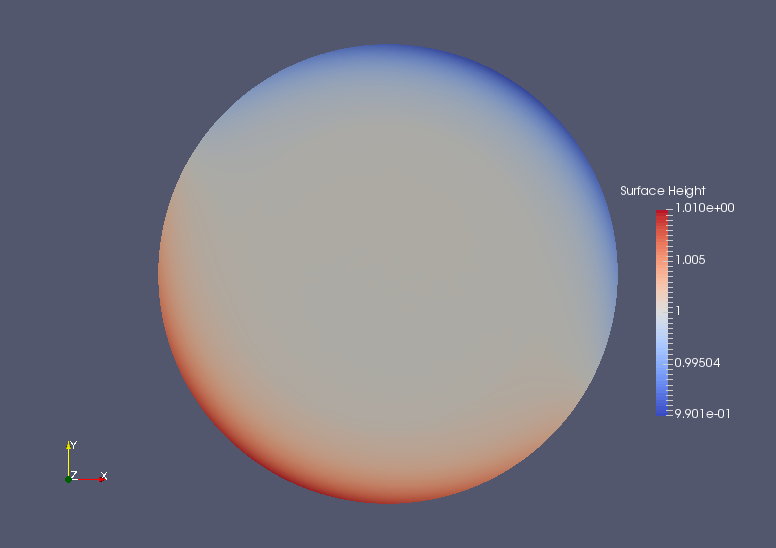}
    }
  \centerline{
    \includegraphics[width=5cm]{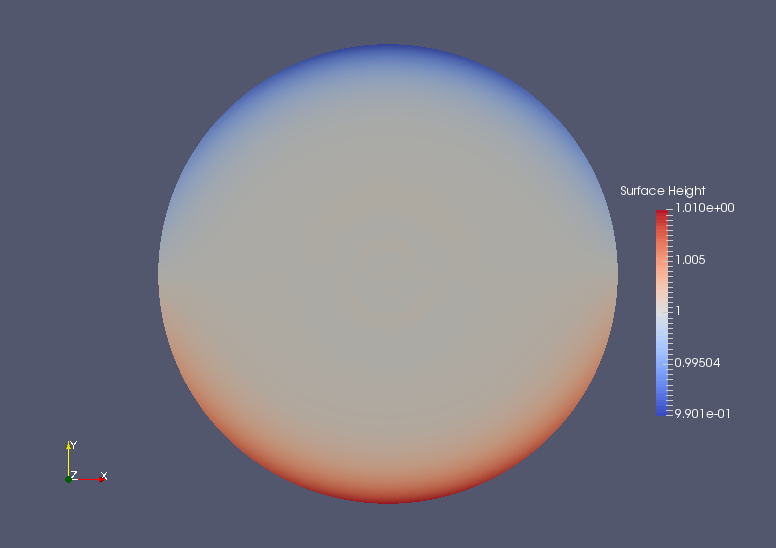}
    \includegraphics[width=5cm]{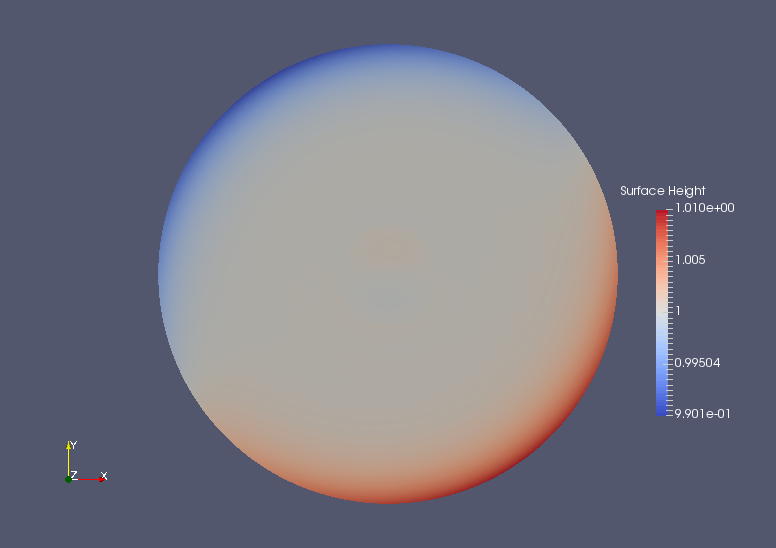}
    \includegraphics[width=5cm]{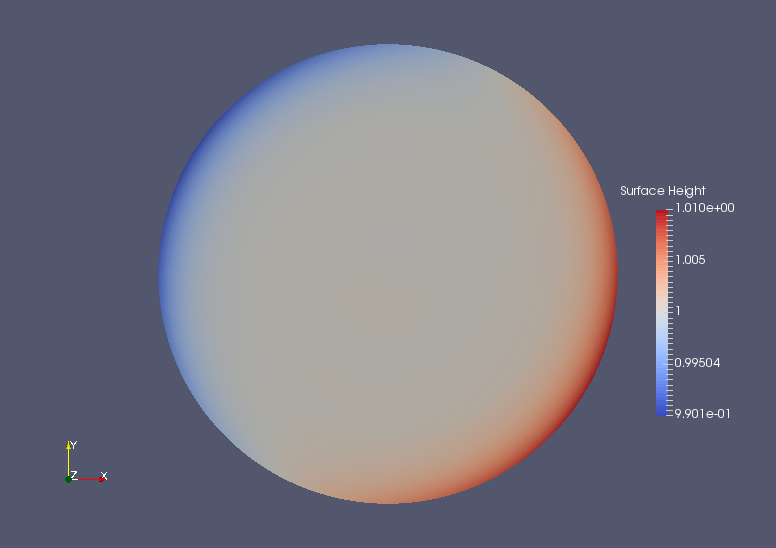}
  }
  \centerline{
    \includegraphics[width=5cm]{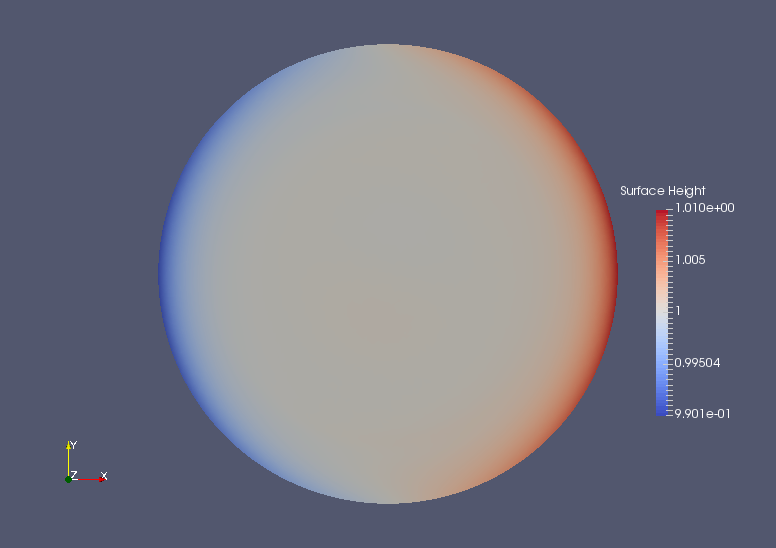}
    \includegraphics[width=5cm]{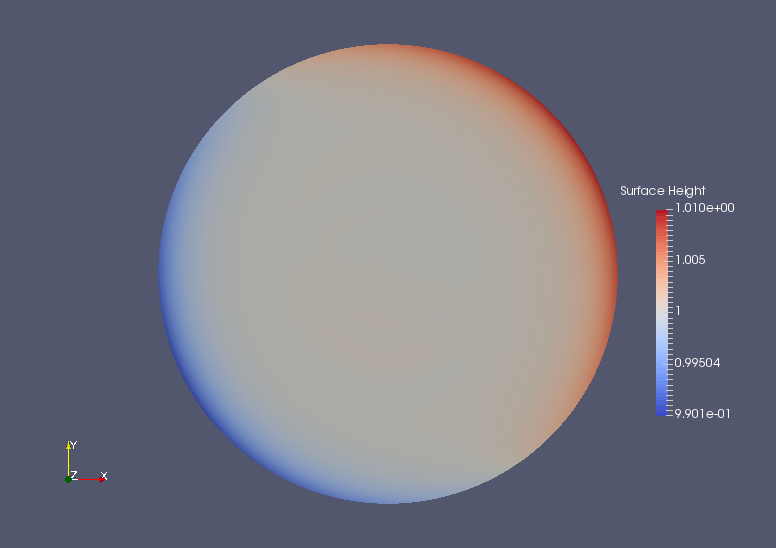}
    }
  \caption{\label{fig:kelvin}Plots showing Kelvin wave propagation in a
    disk with constant $f$, and wave speed 1. We observe the
    near-travelling-wave solution (which is approached as the radius
    of the disk goes to infinity) propagating at this speed. Snapshots
    are shown every 0.5 nondimensional time units until 5 time units. The circumference of the disk is 2$\pi$, so a total rotation of the Kelvin wave should
  take approximately $2\pi$ time units; this is observed.}
\end{figure}}

\subsection{Energy-conserving, enstrophy-dissipating scheme}
For turbulent large scale balanced dynamics, enstrophy conservation is
not necessarily desirable since the enstrophy cascade to small scales
will mean that the enstrophy will accumulate at the gridscale.
\citet{arakawa1990energy} introduced an anticipated potential
vorticity approach, which modifies the equations so that they still
conserve energy but have an upwinded/anticipated potential vorticity
so that enstrophy is dissipated at small scales. Applying the
modification to the energy-enstrophy conserving discretisation of
\citet{arakawa1981potential} preserves these properties at the
discrete level. This approach has been followed by a number of authors
extending the C-grid approach to unstructured grids and compatible
finite element methods. In this paper we modify this approach using a
streamline-upwind Petrov-Galerkin (SUPG) method for the implied
potential vorticity equation. Unlike the anticipated potential
vorticity method, this remains a consistent scheme to the equations
without dissipation, but still introduces enstrophy dissipation at the
gridscale. This approach was originally advocated in
\citet{cotter2014finite,mcrae2015compatible}. The SUPG method, surveyed in
\citet{hughes95}, replaces the test function in the weak formulation
by a modified test function that is biased in the upwind
direction. Since this test function is applied to the entire equation,
this does not alter the residual formulation, only the nature of the
test functions, and hence the scheme is expected to remain consistent
at the appropriate order. SUPG was first applied to the Euler
equations in streamfunction-vorticity formulation by
\citet{tezduyar88,tezduyar89}, and the multiscale behaviour of the
resulting scheme was examined by \citet{natale2017scale}.

The SUPG modification of the energy-enstrophy conserved shallow
water scheme of this paper is obtained by replacing $q^\delta$
in Equation \eqref{eq:bcs u} by $q^*$ given by
\begin{equation}
  q^* = q^\delta - \frac{\tau}{\D}\left((q^\delta D^\delta)_t + \nabla\cdot(
  q^\delta \F)\right),
\end{equation}
where $\tau$ is a chosen time parameter. Equivalently, making use of
the fact that $D_t^\delta+\nabla\cdot\F=0$ in $L^2$, we have
\begin{equation}
  q^* = q^\delta - \tau\left(q^\delta_t + \frac{\F}{D^\delta}\cdot \nabla q^\delta\right).
  \label{eq:q*alt}
\end{equation}
If $q^\delta$, $\F$ and $D^\delta$ are replaced with the exact solutions $q$,
$\MM{F}$ and $D$ respectively, this extra term vanishes, because the residual
of the equation is zero. We see that this is a consistent modification
to the equations. Further, since this replacement is equivalent to
replacing $q^\delta$ with $q^*$ in the Poisson bracket from Definition
\ref{def:extended bracket}, we see that the modified formulation still
conserves energy.

After manipulations identical to the previous section, the equivalent
PV equation is
\begin{equation}
  \langle \gamma, \left(q^\delta D^\delta\right)_t
  \rangle -\left\langle \nabla \gamma,
  \left(q^\delta - \frac{\tau}{D^\delta}\left((q^\delta D^\delta)_t + \nabla\cdot(
  q^\delta \F)\right)\right)\F\right\rangle = 0, \quad \forall \gamma \in V_0,
\end{equation}
which rearranges to
\begin{equation}
  \left\langle \gamma + \frac{\tau}{D^\delta}\F\cdot\nabla\gamma, \left(q^\delta D^\delta\right)_t
  + \nabla\cdot(q^\delta \F)
  \right\rangle = 0, \quad \forall \gamma \in V_0.
\end{equation}
This is the SUPG discretisation of the potential vorticity conservation
equation. Selecting $\gamma=1$ implies that the total potential vorticity
is still conserved.

After using \eqref{eq:q*alt} to rewrite the SUPG formulation as 
\begin{equation}
  \left\langle \gamma + {\tau}\F\cdot \nabla\gamma, q^\delta_t
  + \frac{\F}{D^\delta}\cdot\nabla q^\delta
  \right\rangle = 0, \quad \forall \gamma \in V_0,
\end{equation}
we see that there is the possibility for enstrophy
dissipation since the term
\begin{equation}
  \left\langle \frac{\tau}{D^\delta}\F\cdot \nabla\gamma,
  \F \cdot \nabla q^\delta
  \right\rangle = 0,
\end{equation}
is positive semi-definite, \emph{i.e.} setting $\gamma=q^\delta$ gives a non-negative number. Indeed, this is the term that arises in the
anticipated potential vorticity method translated to compatible finite
element methods in \citet{mcrae2014energy}.

The other SUPG term
\begin{equation}
  \left\langle \tau \F\cdot\nabla\gamma,
  q^\delta_t
  \right\rangle = 0,
\end{equation}
is sign indefinite, and so $\tau$ needs to be sufficiently large to
guarantee monotonic decay in $Z$. The presence of this term means that
the scheme is consistent. When the solution is well-resolved, there
will be no dissipation of enstrophy, which only becomes significant
once the solution becomes marginally-resolved.

In this numerical example, we used a semi-implicit 
timestepping formulation for a more efficient implementation.
Our semi-implicit scheme can be thought of as a fixed number of
Picard iterations towards the energy-conserving time integrator
given above. Writing Equations (\ref{eq:time u}-\ref{eq:time D})
in the form
\begin{equation}
  R_u[\uu^{n+1},\D^{n+1};\MM{w}] = 0, \, \forall \MM{w}\in \mathring{V}_1, \quad
  R_D[\uu^{n+1},\D^{n+1};\phi] = 0, \, \forall \phi\in V_2,
\end{equation}
where any dependency on $\q^{n+1}$ or $\F^{n+1/2}$ is obtained by
solving Equations \eqref{eq:time F} and \eqref{eq:time q} using
$\uu^{n+1}$ and $\D^{n+1}$.

Then to implement the timestep, we perform some number (4, in the case
of this example) of fixed point iterations as follows. In the first
iteration we take the initial guess $\uu^{n+1}=\uu^n$,
$\D^{n+1}=\D^n$. Then we solve for $\Delta \MM{u}\in \mathring{V}_1$,
$\Delta D\in V_2$ such that
\begin{align}
  \label{eq:linear u}
  \langle \Delta\MM{u}, \MM{w} \rangle + \frac{\Delta t}{2}\langle \MM{w},
  f\Delta \MM{u}^\perp \rangle - \frac{\Delta t}{2}\langle \nabla\cdot \MM{w},
  \Delta D\rangle & = - R_u[\uu^{n+1},\D^{n+1};\MM{w}], \quad
  \forall \MM{w}\in \mathring{V}_1, \\
  \label{eq:linear D}
  \langle \Delta D + H\nabla\cdot\Delta\MM{u}, \phi \rangle
  & = - R_D[\uu^{n+1},\D^{n+1};\phi], \quad
  \forall \phi\in \mathring{V}_2.
\end{align}
Then we update $\uu^{n+1} \mapsfrom \uu^{n+1}+\Delta\MM{u}$, $\D^{n+1}\mapsfrom
\D^{n+1}+\Delta D$. This completes one fixed point iteration.

The system (\ref{eq:linear u}-\ref{eq:linear D}) can be hybridised and
statically condensed to obtain a sparse discrete Helmholtz problem
which we assembled using the Slate subpackage of Firedrake 
\citep{gibson18}. We used $\tau=\Delta t$, with $\Delta t=$50s on mesh
refinement level 5, where we observed monotonic decay in the
enstrophy.

In this example, we used the initial condition of test case 5 from
\citet{williamson1992standard}, in which a balanced flow in solid
rotation is disturbed by the sudden appearance of a conical mountain
at time 0. The difference in this example is that the equations are
solved on the hemisphere, not the sphere. This means that a comparison
cannot be made with other integrations, but we can check the stability
of the scheme and observe the energy and enstrophy behaviour for long
integration times once small scale vorticity filaments appear.

Snapshots of the solution are shown in Figure \ref{fig:supg6}. These
solutions show the formation of vortex filaments that roll up and
stretch until they reach the grid scale, where they become more noisy
(but do not pollute the whole domain with oscillations as they do in
the enstrophy-conserving case). Plots of energy and enstrophy
are shown in Figure \ref{fig:w5enens}. As expected we see approximate
energy conservation over all times, and enstrophy eventually starts to
decay when gridscale features appear. We also observed conservation of
total potential vorticity up to round-off error.

\revised{To make a more careful study of the energy conservation, we
  performed a convergence experiment where we ran the same mountain
  test case for 15 days at mesh refinement level 6, to obtain a more
  interesting flow. We then used this as the initial condition for a
  finite time interval convergence test checking for energy
  conservation as $\Delta t$ goes to zero, verifying that the spatial
  discretisation is exactly energy-conserving. This is demonstrated in
  Figure \ref{fig:longrun}.}

\begin{figure}
  \centerline{
    \includegraphics[width=5cm]{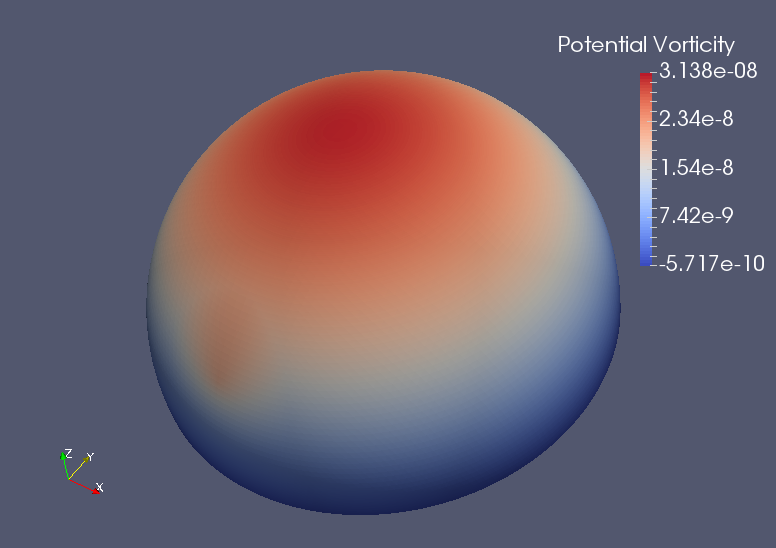}
    \includegraphics[width=5cm]{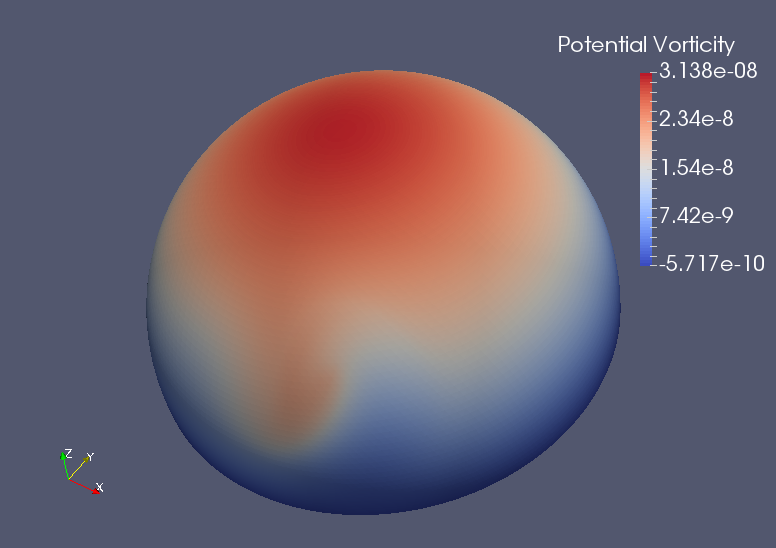}
    \includegraphics[width=5cm]{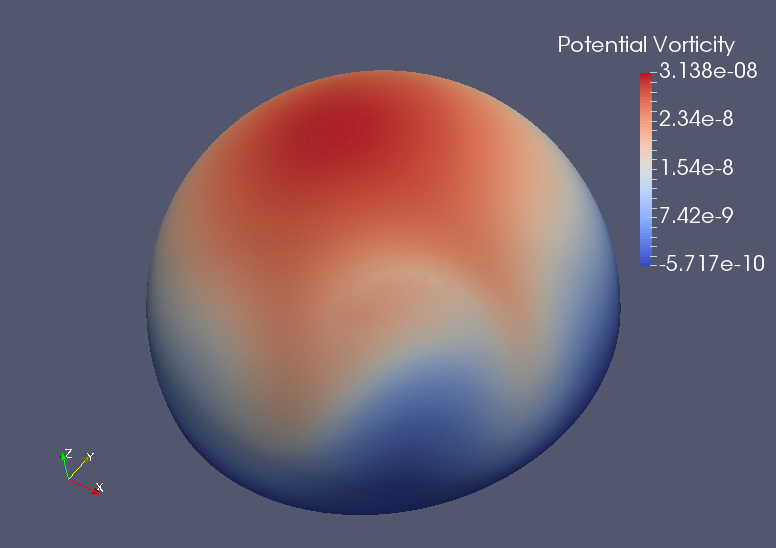}
    }
  \centerline{
    \includegraphics[width=5cm]{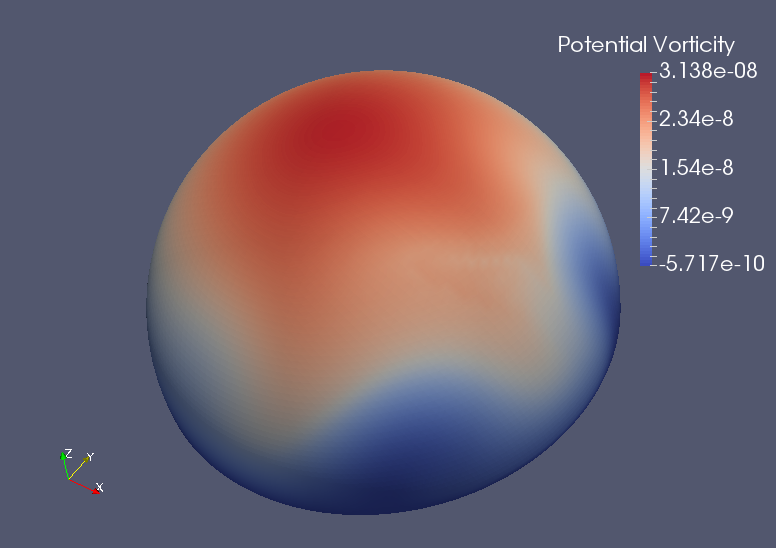}
    \includegraphics[width=5cm]{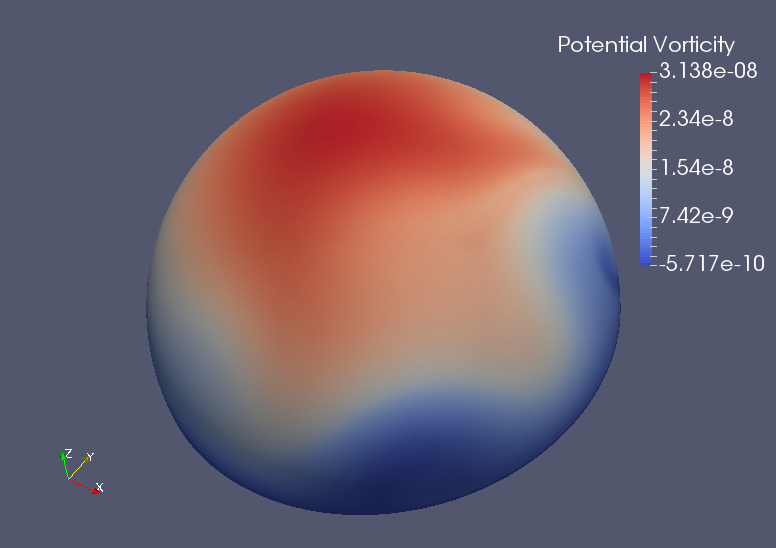}
    \includegraphics[width=5cm]{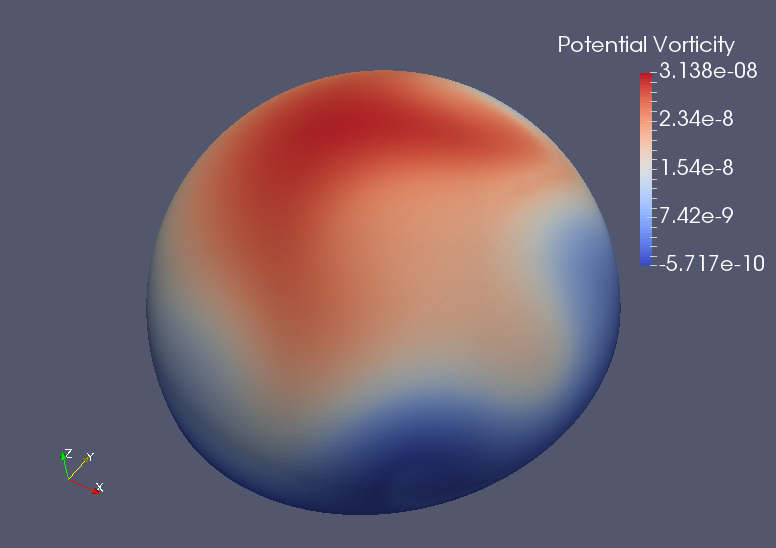}
    }
  \centerline{
    \includegraphics[width=5cm]{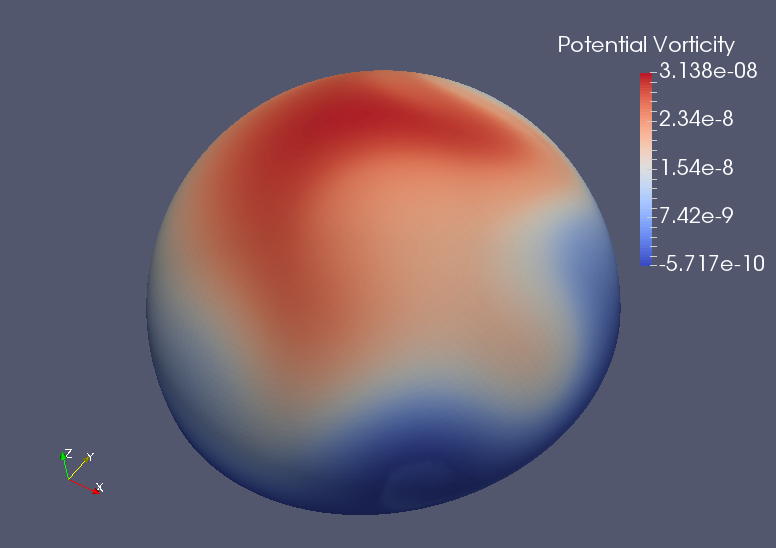}
    \includegraphics[width=5cm]{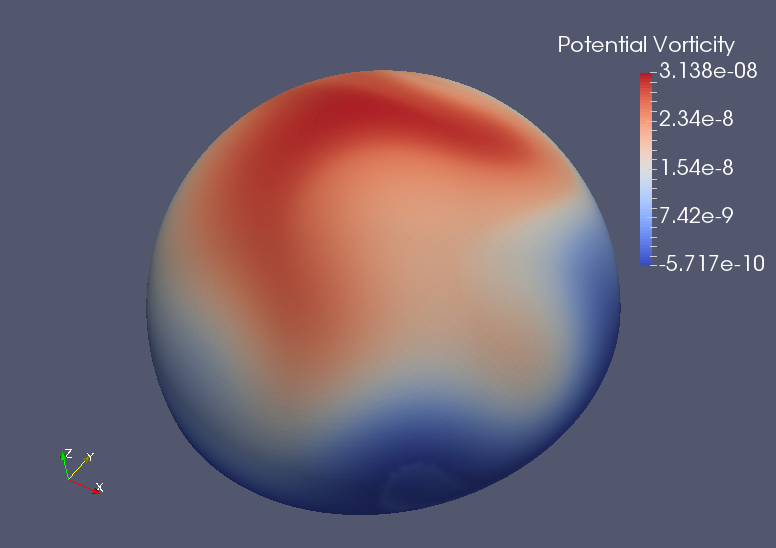}
    \includegraphics[width=5cm]{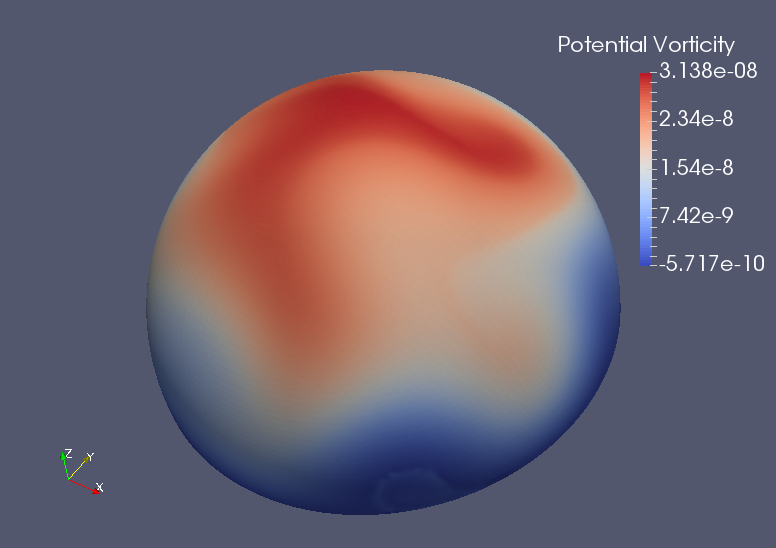}
  }
  \centerline{
    \includegraphics[width=5cm]{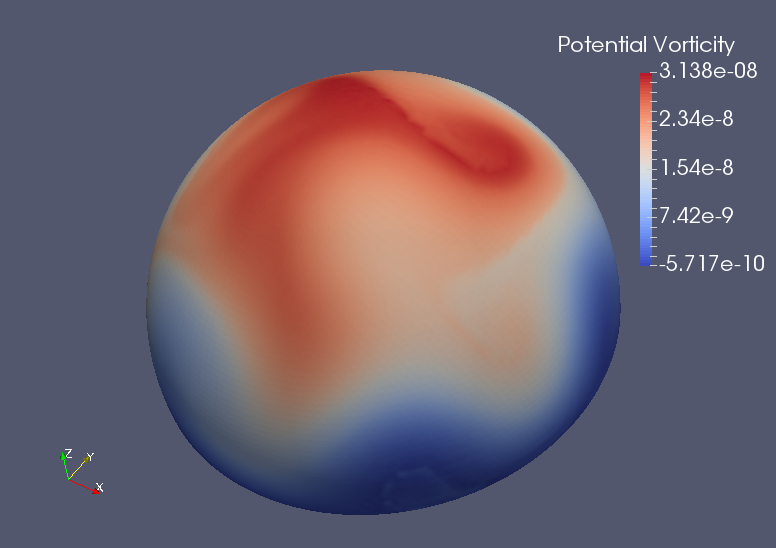}
    \includegraphics[width=5cm]{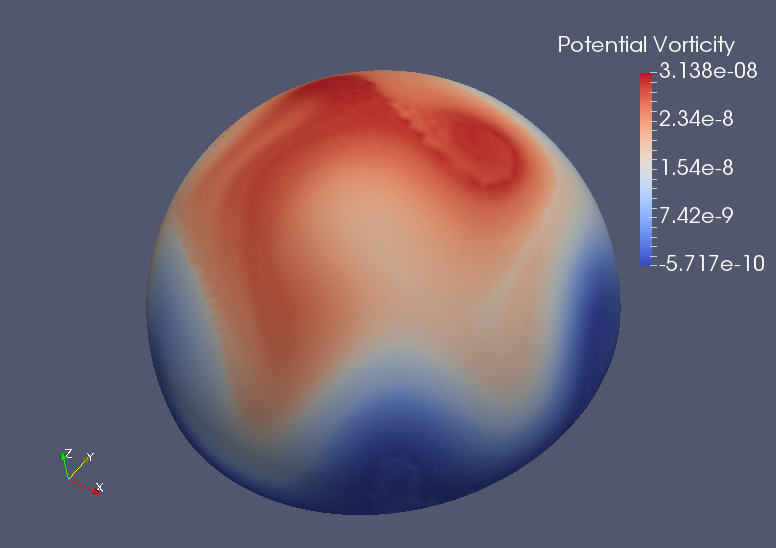}
    \includegraphics[width=5cm]{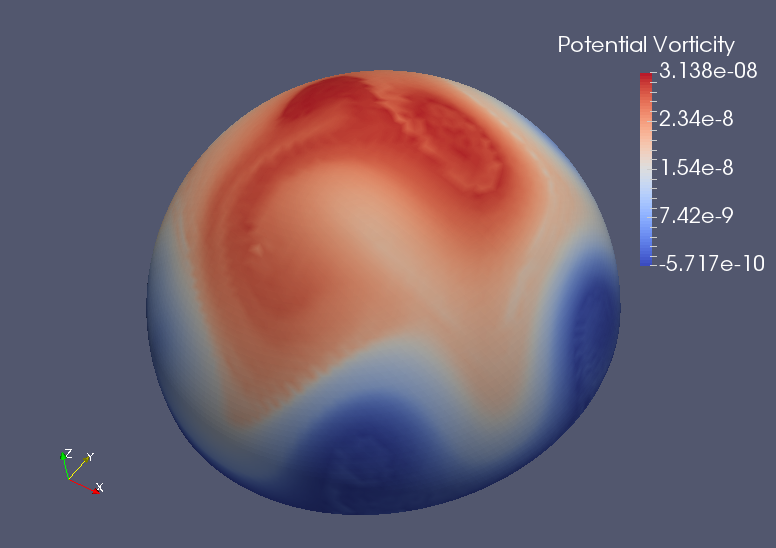}
    }
  \centerline{
    \includegraphics[width=5cm]{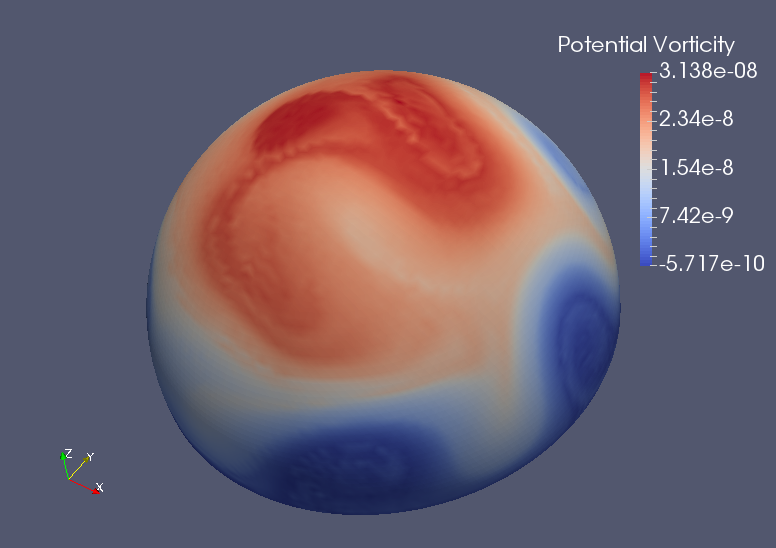}
  }
  \caption{\label{fig:supg6}Snapshots of the hemisphere version of the Williamson 5 testcase showing potential vorticity, at refinement level 6, plotted every 180000 seconds.}
%dt=25, dumpfreq=100, we took every 100 plots = 250000 seconds.
\end{figure}

\begin{figure}
  \centerline{
  \includegraphics[width=12cm]{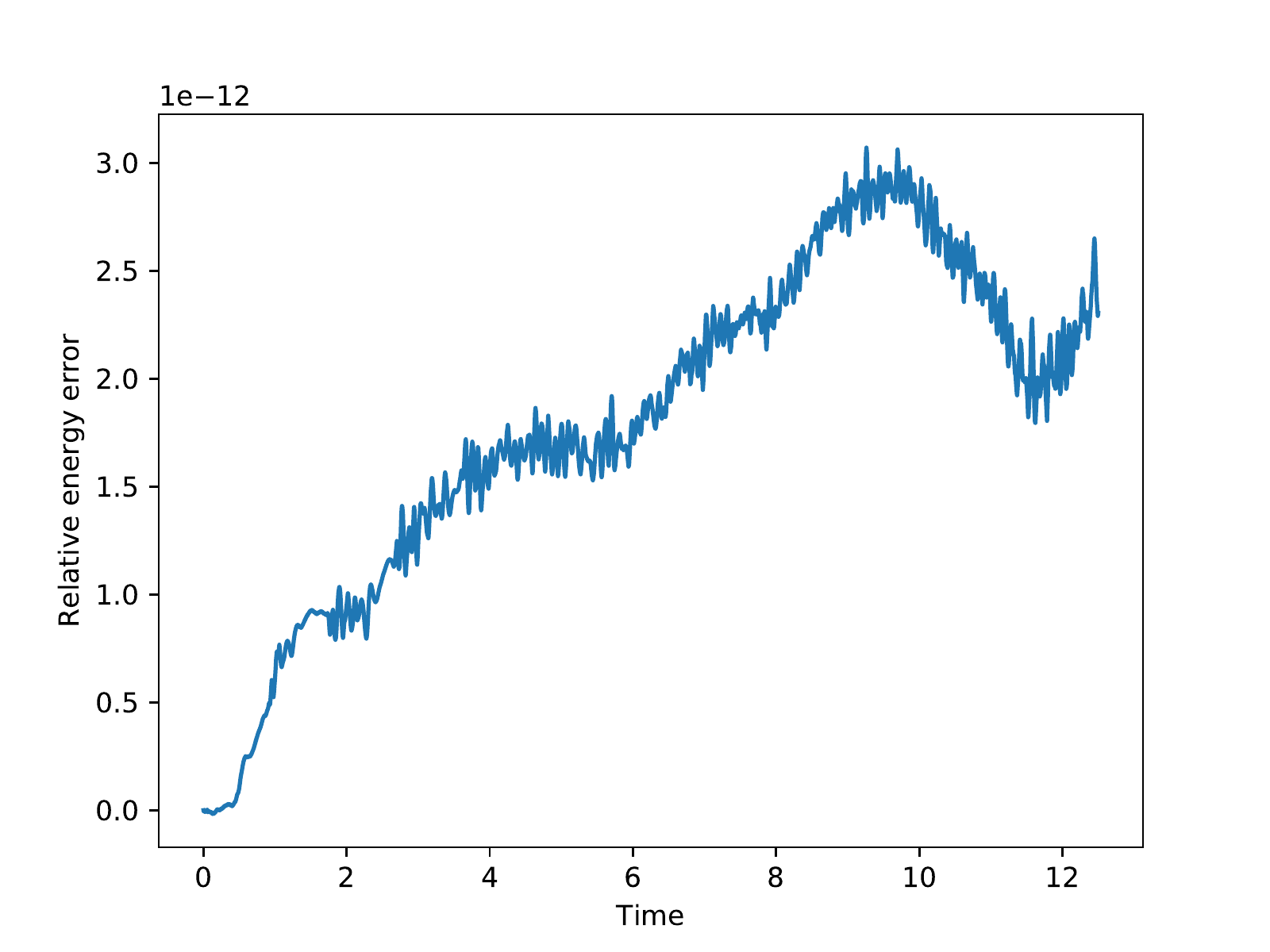}}
  \centerline{
  \includegraphics[width=12cm]{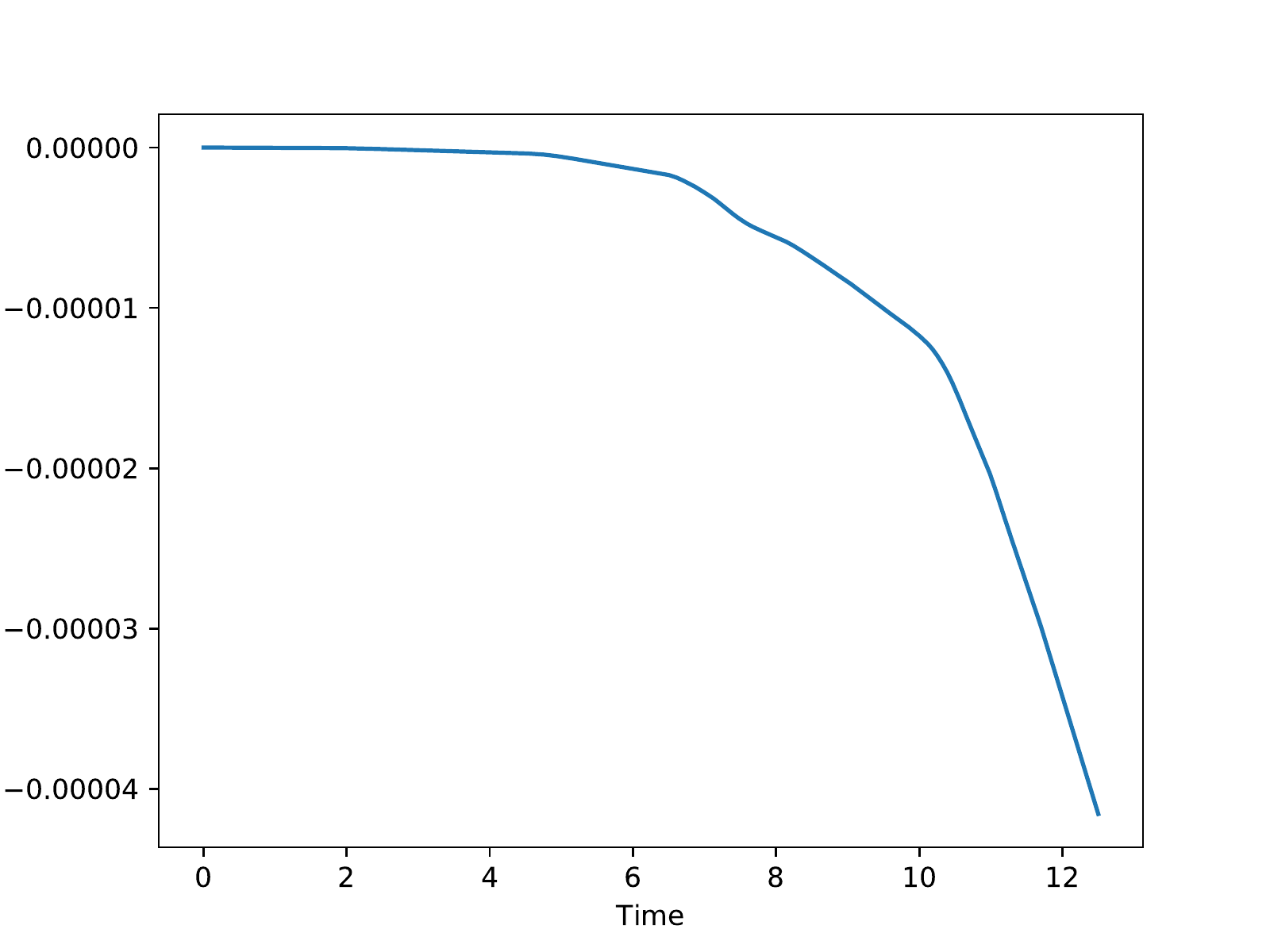}
  }
  \caption{\label{fig:w5enens}Plots showing relative energy and
    enstrophy errors for the Williamson 5 test case (modified to run
    on a hemisphere) at mesh refinement level 6 with $\Delta t=10$. We
    observe approximate energy conservation for all times. We observe
    approximate enstrophy conservation initially, and then decay of
    enstrophy once the solution becomes unresolved.}
\end{figure}

\begin{figure}
  \centerline{
    \includegraphics[width=12cm]{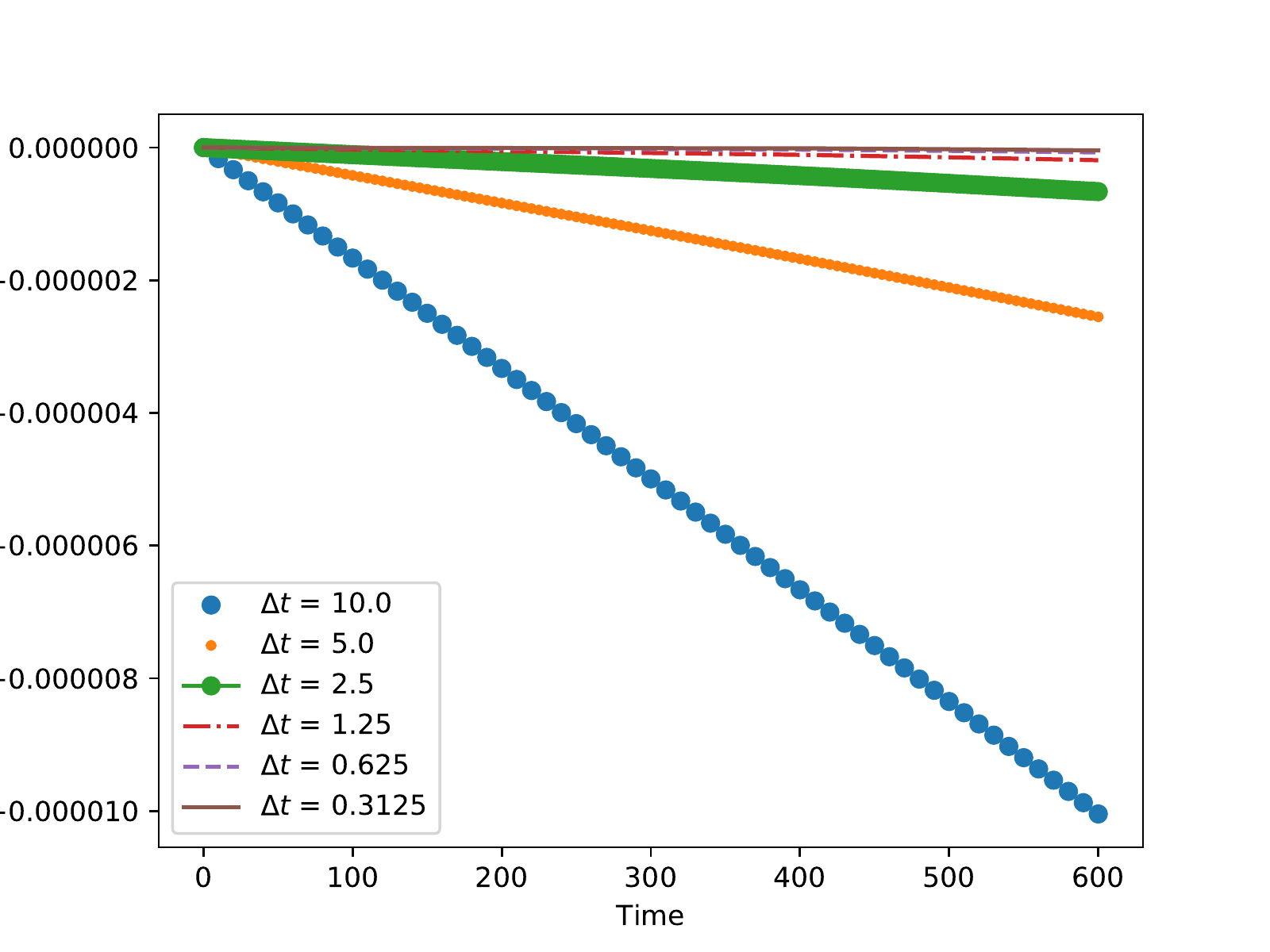}
  }
  \caption{\label{fig:longrun}\revised{Plots showing convergence of energy
    error on a finite time interval to zero as $\Delta t$ is
    reduced. The mountain test was run for 15 days at refinement level
    6, and then this was used as an initial condition for the
    convergence test. This shows that the spatial discretisation is
    exactly preserving energy and any errors are due to the fact that
    the Picard iteration scheme has not fully converged to the Poisson
    integrator.}}
\end{figure}

\section{Summary and outlook}
\label{sec:summary}
In this paper we provided a compatible finite element scheme for the
rotating shallow water equations in the presence of boundaries. The
prognostic variables are velocity and height, plus vorticity on the
boundary. The scheme has an equivalent formulation where one discards
the potential vorticity - velocity relationship (which is then a
consequence of the subsequent equations) in favour of a potential
vorticity evolution equation everywhere in the domain. This becomes
reminiscent of the energy-enstrophy mimetic spectral element
conserving formulation of \citet{palha2017mass}, in which both
vorticity and velocity are maintained as prognostic variables and a
time-staggered discretisation leads to efficient implementation with
energy and enstrophy conservation, although in that case the
correspondence between velocity and vorticity is not guaranteed by the timestepping scheme.  In our numerical experiments we used an implicit
energy-conserving timestepping scheme, which preserves the
correspondence between velocity and vorticity, but does not preserve
enstrophy {exactly}.  This scheme was implemented via Newton's method
which does require the assembly of a Jacobian operator during each
iteration. We observed optimal convergence rates when running the
solid body rotation testcase. We then modified the scheme to an SUPG
enstrophy dissipating scheme which still conserves energy. Using a
more efficient time integration scheme that backs off from energy
conservation, we demonstrated that this
scheme conserves total potential vorticity and has good energy behaviour whilst dissipating
enstrophy once fine vorticity filaments form, as is consistent with
the enstrophy cascade for 2D geostrophic turbulence.

In future work we will explore the development of vorticity-based
schemes for the three-dimensional compressible Euler equations in
the context of numerical weather prediction.

\paragraph{Acknowledgements} The authors would like to thank Chris Eldred
for suggesting the energy-conserving Poisson integrator, and John
Thuburn, Nigel Wood, and Laurent Desbreu for useful discussions.  CJC
acknowledges funding from NERC grant NE/M013634/1. WB acknowledges
funding from the European Union’s Horizon 2020 research and innovation
programme under the Marie Sk{\l}odowska-Curie grant agreement No. 657016.

\bibliography{ee_bcs}
\end{document}